\documentclass{amsart}
\usepackage{amsfonts,amssymb,amsmath,amsthm}
\usepackage{url}
\usepackage{enumerate}
\usepackage[usenames]{color}

\newtheorem{theorem}{Theorem}[section]
\newtheorem{lemma}[theorem]{Lemma}
\newtheorem{proposition}[theorem]{Proposition}
\newtheorem{corollary}[theorem]{Corollary}
\theoremstyle{definition}
\newtheorem{definition}[theorem]{Definition}

\newtheorem{example}[theorem]{Example}
\newtheorem{problem}[theorem]{Problem}
\newtheorem{fact}[theorem]{Fact}
\numberwithin{equation}{section}
\newtheorem{algorithm}[theorem]{Algorithm}
\newtheorem{subroutine}[theorem]{Subroutine}

\newcommand{\bs}{\boldsymbol}

\newcommand{\Q}{{\mathbb Q}}
\newcommand{\F}{{\mathbb F}}
\newcommand{\K}{{\mathbb K}}
\newcommand{\cF}{\overline{\F\!}{\,}}
\newcommand{\cnd}{C_{n,d}}

\renewcommand{\char}{\text{char}}

\let\phi\varphi

\newcommand{\Qone}{\poly^{(1)}}

\newcommand{\poly}{R}
\newcommand{\softo}{O\,\,{\widetilde{}}\,}

\author{Joachim von zur Gathen}
\address{
    B-IT\\
    Universit\"{a}t Bonn\\
    D - 53113 Bonn}
\email{gathen@bit.uni-bonn.de}

\author{Guillermo Matera}
\address{
Universidad Nacional de General Sarmiento, Instituto del Desarrollo
Humano, J.M. Guti\'errez 1150
(B1613GSX) Los Polvorines, Buenos Aires, Argentina \\
and National Council of Science and Technology (CONICET), Argentina}
\email{gmatera@ungs.edu.ar}

\begin{document}

\title{Interpolation by decomposable univariate polynomials}

\begin{abstract}
The usual univariate interpolation problem of finding a monic polynomial $f$ of degree $n$
that interpolates $n$ given values is well understood.
This paper studies a variant where $f$ is required to be composite, say, a composition
of two polynomials of degrees $d$ and $e$, respectively, with $de=n$, and therefore $d+e-1$ given values.
Some special cases are easy to solve, and for the general case,
we construct a homotopy between it and a special case.
We compute a \emph{geometric solution}
of the algebraic curve presenting this homotopy,
and this also provides an answer to the interpolation task.
The computing time is polynomial in the geometric data, like the degree, of this curve.
A consequence is that for almost all inputs, a decomposable
interpolation polynomial exists.
\end{abstract}
\maketitle

\begin{center}\today\end{center}

\section{Introduction}
For two univariate polynomials $g,h \in \F[X]$ of degrees $d$, $e$,
respectively, over a field $\F$, their \emph{composition}
\begin{equation}
\label{eq:composition} f = g(h) = g \circ h \in \F[X]
\end{equation}
is a polynomial of degree $n = de$. If such $g$
 and $h$ exist with degree at least $2$, then $f$ is called
 \emph{decomposable} (or \emph{composed, composite}, a \emph{composition}).
With a suitable normalization, explained in Section \ref{sec:NumberInterpolPols},
we denote as $C_{n,d}(\F)$ the set of decomposable polynomials $f$ as above.

The standard univariate interpolation problem is well understood.
It asks for a monic interpolation polynomial $f$, given $\deg f$ many constraints on values.
When we require $f \in C_{n,d}(\F)$, then only the $d+e-1$ coefficients of $g$ and $h$
need to be determined. It is natural to impose the same number of constraints.
We write $\cF$ for an algebraic closure of $\F$ and
consider the following interpolation problem:
\begin{problem}\label{problem:interpolation}
Given integers $d, e \geq 2$,
$\bs\alpha=(\alpha_1,\ldots,\alpha_{d+e-1})\in \F^{d+e-1}$
with $\alpha_i\not=\alpha_j$ for $i\not=j$,
$\bs\beta=(\beta_1,\ldots,\beta_{d+e-1})\in \F^{d+e-1}$, and $n=de$, does there
exist $f\in\cnd(\cF)$ such that $f(\alpha_i)=\beta_i$ for $1\le i\le
d+e-1$ and $f=g \circ h$ with $\deg g=d$ as above?
\end{problem}

Standard univariate interpolation at $n$ points
$(\alpha_i, \beta_i)$ has a solution if $\alpha_i \neq \alpha_j$ for $i \neq j$,
that is, $\prod_{i\neq j}(\alpha_i-\alpha_j)$ is nonzero.
But most polynomials are indecomposable: the decomposable ones form a subvariety of small dimension
within the space of all polynomials (Fact \ref{th:geometryCnd}).
Thus interpolating at $n$ points and then decomposing will usually not furnish a solution.
Our parametrization provides a polynomial $R$ (Definition \ref{defGeneric}) such that $R(\alpha)\neq 0$ guarantees
a decomposable interpolant.
Thus almost all such problems have a solution and Problem \ref{problem:interpolation} a positive answer.

This issue arises, for example, in the black-box model of representing polynomials.
If it is expensive to calculate any single value of the polynomial---say by some experiment---%
then one may wish to minimize the number of values sufficient to identify the polynomial.
If, in addition, the polynomial is a priori known to be decomposable,
then our method provides the optimal number of values required.

%
\section{Special cases of the interpolation problem}
\label{section:specialBeta}
In this section, we show that Problem
\ref{problem:interpolation} has a positive answer
 for a generic choice of $\bs\alpha\in
\F^{d+e-1}$ and certain, rather special, $\bs\beta\in \F^{d+e-1}$.
These results hold over an arbitrary field $\F$.

For any positive integer $\ell$ and vector $\bs x = (x_1, \ldots, x_\ell) \in \F^\ell$,
we write $\# \bs x$ for the number of different coordinates of $\bs x$,
so that $\# \bs x = \# \{x_1, \ldots, x_\ell\}$.
Then $\bs x$ defines a set partition $\mathcal P_{\bs x}$ of $\{1,\ldots,\ell\}$ into $\# \bs x$ blocks,
where $i,j \leq \ell$ belong to the same block if and only if $x_i = x_j$.
The Stirling number of the second kind
$S_2(\ell,m)$ is the number of set partitions of $\{1,\ldots,\ell\}$ into exactly $m$ blocks
($m$-partitions).
For a polynomial $f \in \F[X]$, we write $f(\bs x) = (f(x_1),
\ldots, f(x_\ell)) \in \F^\ell$, so that $\# f(\bs x)$ equals the
number of different values that $f$ takes on $\{x_1, \ldots,
x_\ell\}$. For two set partitions ${\mathcal P}$ and $ {\mathcal Q}$
we say that ${\mathcal P}$ is a \emph{refinement} of ${\mathcal Q}$
and write ${\mathcal P} \preccurlyeq {\mathcal Q}$, if each block of
${\mathcal P}$ is contained in one block of ${\mathcal Q}$. If
$\mathcal P$ consists of $m$ blocks, we call it an $m$-refinement of
$\mathcal Q$. Then $m \geq \#\mathcal Q$, and some $m$-refinement of
$\mathcal Q$ exists if and only if $\ell \geq m \geq \#\mathcal Q$. This
induces a partial order on the set of partitions of
$\{1,\ldots,\ell\}$. The minimum of this lattice consists of $\ell$
singletons, and its maximum of a single block.
Let  $\bs\alpha=(\alpha_1,\ldots,\alpha_{d+e-1})\in \F^{d+e-1}$ and $(g,h)$ be
a solution to Problem \ref{problem:interpolation}.
Then
\begin{equation*}
\label{implication} h(\alpha_i) = h(\alpha_j) \Longrightarrow
\beta_i = g\circ h(\alpha_i) =g\circ h(\alpha_j) = \beta_j \text{
for } i,j \leq d+e-1.
\end{equation*}
Thus ${\mathcal P}_{h(\bs \alpha)} \preccurlyeq {\mathcal P}_{\bs \beta}$
and $\# h(\bs \alpha) \geq \#\bs \beta$.

We use the polynomial $\Qone$ defined in \eqref{Q1def} below.
\begin{lemma}\label{lemma:cardinality1}
The polynomial $\Qone \in \F[\bs A]$ is nonzero and has degree less
than $e^2 \, S_2(d+e-1,d-1)$. For any
$\bs\alpha=(\alpha_1,\ldots,\alpha_{d+e-1})\in \F^{d+e-1}$ with
$\Qone(\bs \alpha) \neq 0$ and any $h = X^e +h_{e-1}X^{e-1} + \cdots + h_1 X \in \F[X]$
we have
$\#h(\bs\alpha) \geq d$.
\end{lemma}
\begin{proof}
Assume that $h=X^e+h_{e-1}X^{e-1}+\cdots+h_1X\in \F[X]$
satisfies $m = \#h(\bs\alpha) < d$. We consider the
$m$-partition $\mathcal P = \{P_1, \ldots, P_m\} = {\mathcal
P}_{h(\bs \alpha)}$ of $\{1,\ldots,d+e-1\}$. Thus $i,j < d+e$ are in
the same block of $\mathcal P$ if and only if $
h(\alpha_i)=h(\alpha_j)$. In order to facilitate notation, we sort
the integers in each $P_k  \subseteq \{1,\ldots, d+e-1\}$ in the natural order, let $i_k
\in P_k$ be the smallest element of $P_k$, and furthermore sort
$\{1,\ldots,m\}$ according to the $i_k$'s, so that $i_1 < i_2 <
\cdots < i_m$. Then we have
\begin{equation}
\label{oneRelation}
h(\alpha_{i_k}) - h(\alpha_j) = 0 \text{ for all } k\leq m \text{ and } j \in P_k.
\end{equation}
We also set $P_k^{(1)} = P_k \setminus \{i_k\}$.
In particular, if  $P_k = \{i_k\}$ is a singleton, then $P_k^{(1)} = \emptyset$.
The various instances of \eqref{oneRelation} provide
$$
\sum_{1\leq k \leq m} \#P_k^{(1)} = \sum_{1\leq k \leq m} (\#P_k -1) = \sum_{1\leq k \leq m} \#P_k -m
= d+e-1-m \geq e$$
linear equations on the coefficients of $h$.

We now pick $e$ of these equations. We might take any $e$ ones,
but for a unique choice, we take the ``first'' ones, as follows.
We let $\ell \leq m$ be the smallest index so that
$$
\sum_{1\leq k \leq \ell} \# P_k^{(1)} \geq e, \quad P_\ell^{(1)} \neq \emptyset,
$$
let $P_\ell^{(2)} \subseteq P_\ell^{(1)}$ consist
of the first $e- \sum_{1\leq k < \ell} \# P_k^{(1)}$ elements of $P_\ell^{(1)}$,
and $E = \bigcup_{1\leq k < \ell} P_k^{(1)} \cup P_\ell^{(2)}$.
Then $P_\ell^{(2)}$ is nonempty and $\# E = e$.

The equations  \eqref{oneRelation} corresponding to $E$ are:
\begin{equation}
\label{linearEqs}
\left\{
    \begin{array}{ll}
    h(\alpha_{i_k}) - h(\alpha_j)=0 & \text{for } 1\leq k < \ell, \, j \in P_k^{(1)}, \\
    h(\alpha_{i_\ell}) - h(\alpha_j)=0 & \text{for }  j \in P_\ell^{(2)}.
    \end{array}
        \right.
\end{equation}
We denote by $j_0$
the largest value of $j \in P_\ell^{(2)}$.
Then these equations correspond to the following system of $e$ linear equations
for the coefficients of $h$:

\begin{equation}
\label{matrixEquation1}
\left(
    \begin{array}{ccccccc}
      \alpha_{i_1}-\alpha_{j} & \cdots &
      \alpha_{i_1}^e-\alpha_{j}^e\\
      \vdots &  & \vdots \\
      \alpha_{i_{\ell}}-\alpha_{j_0} & \cdots &
      \alpha_{i_{\ell}}^e-\alpha_{j_0}^e\\
    \end{array}
  \right)
\left(
  \begin{array}{c}
    h_1 \\
    \vdots \\
    h_{e-1} \\
    1 \\
  \end{array}
\right)=
\left(
  \begin{array}{c}
    0 \\
     \vdots\\
    0 \\
  \end{array}
\right),
\end{equation}
where the indices correspond to those of \eqref{linearEqs}.
In particular, $j$ has the same value on each line of the matrix, but these values are different for different lines.

If the coefficient matrix is nonsingular, then the system has no
solution $h$, and thus our assumption about $h$ is contradictory. We
can reformulate this by replacing each occurrence of an $\alpha_i$
in the matrix by an indeterminate $A_i$ and letting

\begin{equation}
\label{matrixForP}
\poly_\mathcal P = \text{det}
\left(
    \begin{array}{ccccccc}
      A_{i_1}-A_{j} & \cdots &
      A_{i_1}^e-A_{j}^e\\
      \vdots &  & \vdots \\
      A_{i_{\ell}}-A_{j_0} & \cdots &
      A_{i_{\ell}}^e-A_{j_0}^e\\
    \end{array}
  \right) \in \F[\bs A],
\end{equation}
where $\bs A = (A_1, \ldots, A_{d+e-1})$. Then $\# h({\bs\alpha})
\geq d$ if $\poly_\mathcal P (\bs \alpha) \neq 0$.

We claim that the matrix for  $\poly_\mathcal P$ is nonsingular. For
each of its rows, there is a value $j$ so that $A_j$ occurs on that
line and nowhere else in the matrix, namely as $A_j, \ldots, A_j^e$.
In a Laplace expansion of the determinant, the cofactor of $A_j^e$
is a matrix that also has this property. There is no cancellation
for $A_j^e$ and inductively, it follows that $\poly_\mathcal P \neq
0$. Its total degree is $\sum_{1\leq i \leq e} i = e(e+1)/2 < e^2$.

The partition $\mathcal P$ that gives rise to $\poly_\mathcal P$ is an
$m$-partition of $\{1,\ldots,d+e-1\}$, for some $m < d$. We now
describe a $(d-1)$-partition ${\mathcal Q}$ of $\{1,\ldots,d+e-1\}$
with $\poly_{\mathcal P} = \poly_{{\mathcal Q}}$. We let $E' =
\bigcup_{1\leq k < \ell} P_k  \cup \{i_\ell\} \cup P_\ell^{(2)}$. An
index $i \in \{1,\ldots, d+e-1\}$ occurs somewhere in the matrix for
$\poly_{\mathcal P}$ if and only if $i \in  E'$, and $\#E' = e + \ell$.
Thus there are $d+e-1- (e+\ell) = d - 1 -\ell$ such indices that do
not occur in the matrix. Now ${\mathcal Q}$ consists of the blocks
$P_1,\ldots,P_{\ell-1}$ of ${\mathcal P}$, $\{i_\ell\} \cup
P_{\ell}^{(2)}$, and a singleton $\{i\}$ for all $i \in
\{1,\ldots,d+e-1\} \setminus E'$. Thus ${\mathcal Q}$ has
$\ell + d-1-\ell = d-1$ blocks and is a $(d-1)$-partition of
$\{1,\ldots,d+e-1\}$. Furthermore, the indices in the added
singletons occur neither in the matrix for $\poly_\mathcal P$ nor in
that for $ \poly_{{\mathcal Q}}$, so that $\poly_\mathcal P = \poly_{{\mathcal
Q}}$.

We set
\begin{equation}
\label{Q1def}
\Qone=\prod_{\mathcal{Q}}\poly_{\mathcal{Q}},
\end{equation}
where the product runs over all $(d-1)$-partitions ${\mathcal Q}$ of
$\{1,\ldots,d+e-1\}$.
Then the claim of the lemma follows.
\end{proof}

For $\bs\alpha \in \F^{d+e-1}$, let
$${\mathcal{C}}_{\bs \alpha}=\{h = X^e +h_{e-1}X^{e-1} + \cdots + h_1 X \in \F[X]
\colon \,\#(h(\bs\alpha))\le d\}.$$
We use the polynomial $\poly^{(2)}$ defined in \eqref{Qdef} below.
\begin{lemma}\label{lemma:cardinality2}
The polynomial $\poly^{(2)} \in \F[\bs A]$ is nonzero and has degree less than
$e^2 \, (S_2(d+e-1,d-1) + S_2(d+e-1,d))$.
For any $\bs\alpha \in
\F^{d+e-1}$ with $\poly^{(2)}(\bs \alpha) \neq 0$, we have
$$\#{\mathcal{C}}_{\bs \alpha}=S_2(d+e-1,d),$$
$\#h(\bs\alpha) = d$ for all $h \in {\mathcal{C}}_{\bs \alpha}$,
and $\alpha_i \neq \alpha_j$ whenever $1 \leq i < j < d+e$.
\end{lemma}
\begin{proof}
Let $\bs\alpha \in \F^{d+e-1}$ with $\Qone(\bs \alpha) \neq 0$, and
$h=X^e+h_{e-1}X^{e-1}+\cdots+h_1X\in {\mathcal{C}}_{\bs \alpha}$.
Lemma \ref{lemma:cardinality1}
implies that $\#h(\bs\alpha) = d$.

We now follow the recipe for proving Lemma \ref{lemma:cardinality1}.
Now $\mathcal P = (P_1,\ldots,P_d)$ is a $d$-partition of $\{1,\ldots,d+e-1\}$,
and all instances of \eqref{oneRelation} yield only
$$
\sum_{1\leq k \leq d} \#P_k^{(1)} = \sum_{1\leq k \leq d} (\#P_k -1) = \sum_{1\leq k \leq d} \#P_k -d
= d+e-1-d = e -1$$
linear equations on the coefficients of $h$.
Thus \eqref{matrixEquation1} would provide an $(e-1) \times e$ system of linear equations.
In order to obtain a square coefficient matrix, we move the quantities corresponding to the last column
to the ``constant'' side and obtain:

\begin{equation}
  \label{matrixEquation2}
\left(
    \begin{array}{ccccccc}
      \alpha_{i_1}-\alpha_{j} & \cdots & \alpha_{i_1}^{e-1}-\alpha_{j}^{e-1}  \\
      \vdots &  & \vdots \\
      \alpha_{i_{d}}-\alpha_{j} & \cdots & \alpha_{i_{d}}^{e-1}-\alpha_{j}^{e-1} \\
    \end{array}
  \right)
\left(
  \begin{array}{c}
    h_1 \\
    \vdots \\
    h_{e-1} \\
  \end{array}
\right)= \left(
  \begin{array}{c}
 \alpha_{j}^e-\alpha_{i_1}^e\\
     \vdots\\
   \alpha_{j}^e-\alpha_{i_{d}}^e \\
  \end{array}
\right),
\end{equation}
where again the value of $j$ is constant along each row of the matrix, but different for different rows.

We consider the matrix in $\F[\bs A]^{(e-1) \times (e-1)}$, where
each entry $\alpha_i$ in the coefficient matrix is replaced by the
indeterminate $A_i$ and let $S_{\mathcal P} \in \F[\bs A]$ be its
determinant. By the same argument as for
$\poly_{\mathcal P}$ in the previous proof, also $S_{\mathcal P}$ is
nonzero. Its degree is $(e-1)e/2 < e^2$. If $S_{\mathcal P} (\bs
\alpha) \neq 0$, then the coefficients of $h$ are uniquely
determined by \eqref{matrixEquation2}.

An $h \in  {\mathcal{C}}_{\bs \alpha}$ determines its partition
${\mathcal{P}}$ uniquely via \eqref{oneRelation}. On the other hand, if
$S_{\mathcal P} (\bs \alpha) \neq 0$, then \eqref{matrixEquation2}
determines $h$ uniquely. Thus we have a bijection between
${\mathcal{C}}_{\bs \alpha}$ and the set of $d$-partitions
of $\{1,\ldots,d+e-1\}$, both finite of size $S_2(d+e-1,d)$.

Recalling $\poly^{(1)}$ from \eqref{Q1def}, we let
\begin{equation}
\label{Qdef}
 \poly^{(2)} = \poly^{(1)} \cdot \prod_{\mathcal{P}}S_{\mathcal{P}}
 \in \F[\bs A],
 \end{equation}
 where the product runs over the set of $d$-partitions of
$\{1,\ldots,d+e-1\}$.

 For the last claim, let $1 \leq i < j < d+e$ and choose a $d$-partition $\mathcal P$ of $\{1,\ldots,d+e-1\}$,
 one of whose blocks, say $P_k$, is $\{i,j\}$. Since $e\geq 2$, there exists a $(d-1)$-partition on the other
 $d+e-3 \geq d-1$ elements.
 In the matrix of indeterminates defining $S_{\mathcal P}$, $A_i - A_j$ divides all entries in the row corresponding to $P_k$,
and hence also $S_{\mathcal P}$ and $ \poly^{(2)}$ are divisible by this difference.
Thus  $\alpha_i \neq \alpha_j$ whenever  $ \poly^{(2)}(\bs \alpha) \neq 0$.
 The claims of the lemma follow.
 \end{proof}

 The Bell number $B_n$ counts the number of all set partitions of a set with $n$ elements.
 The estimate in Berend \&\ Tassa \cite{bertas10} yields
 \begin{eqnarray}
 \label{StirlingBound}
 S_2(d+e-1,d-1) + S_2(d+e-1,d)  & < & B_{d+e-1} < \bigg( \frac {d+e-1} {\ln (d+e)} \bigg)^{d+e-1},\nonumber\\
 \deg  \poly^{(2)}  & < & e^2 B_{d+e-1} \leq (d+e)^{d+e-1}.
 \end{eqnarray}

For certain $\bs \beta$ with sufficiently few distinct entries,
we can compute a composite interpolation polynomial $f=g\circ h$ with $\deg g =d$, and the number of such polynomials.

\begin{proposition}
\label{interpolForSmallSizeBeta} Let $\bs \alpha\in\F^{d+e-1}$
satisfy $ R^{(2)} (\bs \alpha) \neq 0$ and $\#\bs \beta \leq d$.
Then we can compute a solution to Problem
\ref{problem:interpolation} for $(\bs \alpha, \bs \beta)$.
All solutions are defined over $\F$ and the number of solutions
equals the number of $d$-refinements of ${\mathcal P}_{\bs \beta}$.
\end{proposition}

\begin{proof}
We let $\mathcal P$ be an arbitrary $d$-refinement
 of $\mathcal P_{\bs \beta}$,
that is, $ \mathcal P \preccurlyeq \mathcal P_{\bf \beta}$. Solving
the nonsingular system \eqref{matrixEquation2} of linear equations
corresponding to $\mathcal P$ provides a solution
$(h_1,\ldots,h_{e-1})$ with all $h_i \in \F$, and we set
$h=X^e+h_{e-1}X^{e-1}+\cdots + h_1X \in \F[X]$. Then $\#h(\bs
\alpha) = d$ and for $1\leq k \leq d$, we let $i_k \in P_k \subset \{1,\ldots,d+e-1\}$ be
a representative of the $k$th block $P_k$ of $\mathcal P$. The
condition that $g(h(\alpha_{i_k})) = \beta_{i_k}$ for $k \leq d$
defines a unique monic interpolation polynomial $g\in\F[X]$ of
degree $d$, and $(g,h)$ is a solution to Problem
\ref{problem:interpolation}.

The proof of Lemma \ref{lemma:cardinality2} notes a bijection
between the set of $h$ and that of the above partitions. This
implies the claim on the number of solutions.
\end{proof}

\begin{example}
\label{example1}
For $d=e=2$, we consider the arbitrarily chosen $\bs \alpha = (5,6,7)$ and $\bs \beta=(3,3,3)$.
Then $\mathcal P_{\bs \beta} = \{1,2,3\}$ has the three 2-refinements
$\{\{1,2\},\{3\}\}$,  $\{\{1,3\},\{2\}\}$, and  $\{\{2,3\},\{1\}\}$.
We find the corresponding three components $g=x^2+g_1x+g_0$ and $h=x^2+h_1x$ with $(g_1, g_0, h_1)$ as
$$
(58, 843, -11), (71, 1263, -12), (82, 1683, -13)
$$
 and their compositions $f = g \circ h$ as
\begin{eqnarray*}
& x^4-22x^3+179x^2-638x+843, x^4-24x^3+215x^2-852x+1263,\\ & x^4-26x^3+251x^2-1066x+1683.
\end{eqnarray*}
This works over any field $\F$ with $\char\ \F \neq 2$.
$\F_2$ does not have enough elements for such an example, but $\F_4 = \F_2[y]/(y^2+y+1)$ does.
For $\bs \alpha = (0,1,y)$ and $\bs \beta = (y+1,y+1,y+1)$, we find the following three
coefficient vectors of components $(g,h)$:
$$
(1,y+1,1), (y+1, y+1,y), (y, y+1, y+1)
$$
For all three, the composition is $f = x^4+x+y+1$.
This is a \emph{composition collision} of a type which is classified in Blankertz et al.\ \cite{blagat13}. \hfill \qed
\end{example}

\begin{corollary}\label{coro:zero_dimensional_fiber_pi}
For an $\bs\alpha\in \F^{d+e-1}$ with $R^{(2)} \neq 0$, the set of solutions $(g,h)$ to
Problem \ref{problem:interpolation} for $(\bs\alpha,\bs 0)$ is finite of
cardinality $S_2(d+e-1,d)$ and all solutions are defined over $\F$.
For such a solution, we have $\# h(\bs \alpha) = d$.
\end{corollary}
\begin{proof}
${\mathcal P}_{\bs 0}$ is the maximum partition in the lattice of partitions,
$\#{\mathcal P}_{\bs 0} = 1$, and it has $S_2(d+e-1,d)$ many $d$-refinements.
\end{proof}
%
%
\section{The number of interpolants}
\label{sec:NumberInterpolPols} We have shown with methods from
linear algebra how to solve the interpolation problem for special
$\bs\beta$. The general case requires tools from algebraic geometry.
We use notions like varieties and morphism only over algebraically closed fields,
but they may be ``defined over'' some smaller field $\F$ with algebraic closure $\cF$.
Standard notions and notations of algebraic geometry can be found in,
e.g., Kunz \cite{Kunz85} or Shafarevich \cite{Shafarevich94}.

We now set some notation for the decomposable polynomials as in \eqref{eq:composition}.
We may assume all three polynomials
to be
monic (leading coefficient 1) and $h$ original (constant coefficient
0, so that the graph contains the origin). All other compositions
can be obtained from this special case by composing (on the left and
on the right) with linear polynomials (polynomials of degree 1);
see, e.g., von zur Gathen \cite{gat14}.

Thus we consider for a proper divisor $d$ of $n$ and $e=n/d$
\begin{align}
\label{CndDef}
P_n(\F) = & \; \{ f \in{ \F}[X] \colon \deg f = n, f \text{ monic}\}, \nonumber\\
P_e^*(\F) = & \; \{ h \in P_e(\F) \colon h \text{ original}\}, \nonumber\\
\gamma_{n,d} \colon & P_d(\F) \times P_e^*(\F) \rightarrow
P_n(\F) \text{
  with }
\gamma_{n,d} (g,h) = g \circ h, \\
C_{n,d}(\F) = & \; \{ f \in P_n(\F) \colon \exists\, g,h \in
P_d(\F) \times
P_e^*(\F) \;\; f = g \circ h \}  \nonumber \\
= & \; \text{im } \gamma_{n,d} , \nonumber \\
C_n(\F) = & \bigcup_{{d \mid n} \atop {d \not \in \{1,n\} } }
C_{n,d}(\F). \nonumber
\end{align}
We may endow $P_n(\F)$ with a structure of $n$-dimensional vector
space over $\F$, associating each $f=X^n+f_{n-1}X^{e-1}+\cdots+f_0$
with the vector $(f_{n-1},\ldots,f_0)$, and $C_n(\F)$ is the
algebraic variety of decomposable polynomials. We drop the argument
$\F$ when it is clear from the context. When $n$ is prime, then
$C_n(\F)$ is empty, and in the following we always assume $n$ to be composite.
We have the following geometric description of $\cnd$.

\begin{fact}[{von zur Gathen \&\ Matera \cite[Theorem
2.2]{gatmat17}}]\label{th:geometryCnd}
  Let $d$ be a proper divisor of $n$ and $\F$ be an algebraically closed field with
  $\char(\F)$ not dividing $d$. Then $\cnd(\F) = ~\mathrm{im}\
  \gamma_{n,d}$ as in \eqref{CndDef} is a closed irreducible algebraic subvariety of
  $P_n(\F)$ of dimension $d+e-1$
and degree at most $d^{d+e-1}$.
\end{fact}

In all of the following, we
assume that $\char(\F)$ does not divide $d$. Our goal is to show that
Problem \ref{problem:interpolation} has an affirmative answer for a
suitably generic instance $(\bs\alpha, \bs\beta)\in \cF^{d+e-1}\times
\cF^{d+e-1}$, namely there exists an interpolation polynomial in
$\cnd$. For this purpose, we introduce the following incidence
variety:
\begin{equation}\label{defGamma}
\Gamma_{n,d}=\{(g,h,\bs\alpha,\bs\beta)\in P_d(\cF)\times P_e^*(\cF)\times \cF^{d+e-1}\times
\cF^{d+e-1} \colon g\circ h(\bs\alpha)=\bs\beta\}.
\end{equation}
We have the following result.
\begin{lemma}\label{lemma:geometryGamma_n,d}
$\Gamma_{n,d}$ is defined over $\F$, irreducible of dimension $2(d+e-1)$
and degree at most $n^{d+e-1}$.
\end{lemma}
\begin{proof}
Let $\bs A=(A_1,\ldots,A_{d+e-1})$, $\bs B=(B_1,\ldots,B_{d+e-1})$,
$\bs G=(G_{d-1},\ldots,G_0)$ and $\bs H=(H_{e-1},\ldots,H_1)$ be
vectors of indeterminates over $\cF$. We denote by $\bs A$ and $\bs
B$ the coordinates of $\cF^{d+e-1}\times \cF^{d+e-1}$, and by $\bs
G$ and $\bs H$ the coordinates of $P_d(\cF)\times P_e^*(\cF)$. Let
$G,H\in \F[\bs G, \bs H, X]$ be the polynomials
$$G=X^d+G_{d-1}X^{d-1}+\cdots+G_0,\quad H=X^e+H_{e-1}X^{e-1}+\cdots
+H_1X,$$
and consider the $\F$-algebra morphism
$$\Phi \colon \F[\bs A,\bs B,\bs G,\bs H]\to \F[\bs A,\bs G,\bs H]$$
defined by $\Phi(B_i)=(G\circ H)(A_i)$ for $1\le i\le d+e-1$. It is
clear that $\Phi$ is surjective, and its kernel is
\begin{align*}
\mathrm{ker}\,\Phi&=\langle B_1-(G\circ H)(A_1),\ldots,
B_{d+e-1}-(G\circ
H)(A_{d+e-1})\rangle=\mathcal{I}(\Gamma_{n,d}),\end{align*}
where $\mathcal{I}(\Gamma_{n,d})\subset\F[\bs A,\bs B,\bs G,\bs H]$
denotes the ideal of polynomials vanishing on $\Gamma_{n,d}$. It
follows that
$$\overline{\Phi}:\F[\bs A,\bs B,\bs G,\bs H]/
\mathcal{I}(\Gamma_{n,d})\to \F[\bs A,\bs G,\bs H]$$
is an isomorphism. Since $\F[\bs A,\bs G,\bs H]$ is a domain,
$\Gamma_{n,d}$ is irreducible, and its dimension agrees with the
Krull dimension of $\F[\bs A,\bs G,\bs H]$, namely $2(d+e-1)$.

Since $\Gamma_{n,d}$ is defined by $d+e-1$ equations
 of degree at most $n=de$, the B\'ezout inequality (see, e.g., Heintz
\cite{Heintz83} or Fulton \cite{Fulton84}) implies the claimed upper bound on its degree.
\end{proof}

We have a natural projection
\begin{equation}\label{def:projection}
\pi \colon \Gamma_{n,d}\to \cF^{d+e-1}\times \cF^{d+e-1}.
\end{equation}
In these terms, Problem \ref{problem:interpolation} can be rephrased
as follows:
\begin{problem}\label{problem:interpolation2}
Given $\bs\alpha=(\alpha_1,\ldots,\alpha_{d+e-1})\in \cF^{d+e-1}$
with $\alpha_i\not=\alpha_j$ for $i\not=j$, and
$\bs\beta=(\beta_1,\ldots,\beta_{d+e-1})\in \cF^{d+e-1}$, is
the fiber $\pi^{-1}(\bs\alpha,\bs\beta)$ nonempty?
\end{problem}

Our next result shows that Problem \ref{problem:interpolation2} has
a positive answer for most $(\bs\alpha,\bs\beta)\in
\cF^{d+e-1}\times \cF^{d+e-1}$.
\begin{corollary}\label{coro:ExistenceGenericInterpolant}
The projection mapping $\pi \colon \Gamma_{n,d}\to \cF^{d+e-1}\times
\cF^{d+e-1}$ is dominant. In particular, if $R^{(2)}(\bs \alpha) \neq 0$,
then the fiber
$\pi^{-1}(\bs\alpha,\bs\beta)$ is of dimension zero.
\end{corollary}
\begin{proof}
Lemma \ref{lemma:geometryGamma_n,d} shows that $\Gamma_{n,d}$ is
irreducible of dimension $2(d+e-1)$. Consider $\pi \colon \Gamma_{n,d}\to
\overline{\pi(\Gamma_{n,d})}$ and observe that both $\Gamma_{n,d}$
and $\overline{\pi(\Gamma_{n,d})}$ are irreducible varieties, where
$\overline{\pi(\Gamma_{n,d})}$ is the Zariski closure of
$\pi(\Gamma_{n,d})$ in $\cF^{d+e-1}\times \cF^{d+e-1}$. For an
$\bs\alpha\in \cF^{d+e-1}$ with  $R^{(2)}(\bs \alpha) \neq 0$,
the fiber
$\pi^{-1}(\bs\alpha,\bs 0)$ has dimension zero by Corollary \ref{coro:zero_dimensional_fiber_pi}.
 Therefore, the theorem on the dimension of fibers (see, e.g., \cite[\S 6.3, Theorem
7]{Shafarevich94}) shows that
$$2(d+e-1)-\dim\overline{\pi(\Gamma_{n,d})}=\dim\Gamma_{n,d}-\dim\overline{\pi(\Gamma_{n,d})}\le 0.$$
It follows that $2(d+e-1)\le \dim\overline{\pi(\Gamma_{n,d})}$, that
is, $\overline{\pi(\Gamma_{n,d})}=\cF^{d+e-1}\times \cF^{d+e-1}$.
This proves the first assertion. Finally, taking into account that
$\Gamma_{n,d}$ and $\cF^{d+e-1}\times \cF^{d+e-1}$ are both
irreducible varieties of dimension $2(d+e-1)$, the second assertion
is a standard consequence of the theorem on the dimension of fibers.
\end{proof}

We may paraphrase the statement of Corollary
\ref{coro:ExistenceGenericInterpolant} in terms of the interpolation
problem in the following way: for any $(\bs\alpha,\bs\beta)$ with $R^{(2)}(\bs \alpha) \neq 0$,
there is a finite positive number of solutions $f=g\circ
h\in\cnd$ to the interpolation problem defined by
$(\bs\alpha,\bs\beta)$. Now we discuss the possible number of
interpolants.

Since the projection
$$\pi \colon \Gamma_{n,d}\to \cF^{d+e-1}\times \cF^{d+e-1}$$
is a dominant mapping defined over $\F$, it induces a finite field
extension $\F(\bs A,\bs B)\hookrightarrow \F(\Gamma_{n,d})$ between
the fields of rational functions on $\cF^{d+e-1}\times \cF^{d+e-1}$
and on $\Gamma_{n,d}$ defined over $\F$. The degree of this field
extension is the {\em degree of the morphism} $\pi$ and is denoted
by $\deg\pi$.

Assume that the field extension $\F(\bs A,\bs
B)\hookrightarrow\F(\Gamma_{n,d})$ is separable (in Proposition
\ref{prop:FieldExtensionFGammaIsSeparable} below we prove this
assertion). By, e.g., \cite[Proposition 1]{Heintz83},
$\#\pi^{-1}(\bs \alpha,\bs\beta)\le\deg\pi$ for any
$(\bs\alpha,\bs\beta)\in \cF^{d+e-1}\times \cF^{d+e-1}$ with a
finite fiber, with equality in a Zariski open subset of
$\cF^{d+e-1}\times \cF^{d+e-1}$. In particular, from Corollary
\ref{coro:zero_dimensional_fiber_pi} we deduce that
$$\deg\pi\ge S_2({d+e-1,d}).$$

We also have an easy upper bound.
\begin{lemma}\label{lemma:UpperBoundDegPi}
We have $\deg\pi\le \deg\cnd\le d^{d+e-1}$.
\end{lemma}
\begin{proof}
Let $(\bs \alpha,\bs\beta)\in \cF^{d+e-1}\times \cF^{d+e-1}$ with
$\#\pi^{-1}(\bs \alpha,\bs\beta)=\deg\pi$. Then $\pi^{-1}(\bs
\alpha,\bs\beta)$ is isomorphic to the zero-dimensional variety
$$\{f \in \cnd\colon f(\alpha_i)=\beta_i\ \text{ for } 1\le i\le d+e-1 \},$$
which is the intersection of $\cnd$ with an affine linear subspace
of $\cF^n$. The B\'ezout inequality (see, e.g.,
\cite{Heintz83} or \cite{Fulton84}) implies the first
inequality, and the second one follows from Theorem
\ref{th:geometryCnd}.
\end{proof}

Corollary \ref{coro:zero_dimensional_fiber_pi} and Lemma
\ref{lemma:UpperBoundDegPi} imply that
\begin{equation}
\label{piBounds}
S_2({d+e-1,d})\le\deg\pi\le d^{d+e-1}.
\end{equation}
To compare both sides of the inequality, we recall that
$d!\,S_2(d+e-1,d)$ is the number of surjective functions from
$\{1,\ldots,d+e-1\}$ to $\{1,\ldots,d\}$ (see, e.g., Mez\"o \cite[page
17]{Mezo20}), while $d^{d+e-1}$ is the total number of functions
from $\{1,\ldots,d+e-1\}$ to $\{1,\ldots,d\}$.

%
%
Next we want to show that $\#\pi^{-1}(\bs \alpha,\bs\beta)=\deg\pi$ for most inputs.
For this purpose, we have to understand
 the ramification
locus of the projection $\pi \colon \Gamma_{n,d}\to \cF^{d+e-1}\times
\cF^{d+e-1}$ and the separability of the corresponding field extension. For $(g,h,\bs\alpha,\bs
\beta)\in\pi^{-1}(\bs\alpha,\bs \beta)$, $\pi$ is {\em unramified}
at $(g,h,\bs\alpha,\bs\beta)$ if the differential mapping
$d_{(g,h,\bs\alpha,\bs\beta)}\pi \colon T_{(g,
h,\bs\alpha,\bs\beta)}\Gamma_{n,d}\to
T_{(\bs\alpha,\bs\beta)}(\cF^{d+e-1}\times \cF^{d+e-1})$ is
injective.

Since the polynomials $Q_i=(G\circ H)(A_i)-B_i$ with $1\le i\le d+e-1 $
defining the incidence variety $\Gamma_{n,d}$ generate a radical
ideal, it follows that the tangent space $T_{(g,
h,\bs\alpha,\bs\beta)}$ of any
$(g,h,\bs\alpha,\bs\beta)\in\Gamma_{n,d}$ is the kernel of the
Jacobian matrix $(\partial \bs Q/(\partial(\bs G,\bs H,\bs A,\bs
B))(g,h,\bs\alpha,\bs\beta)$ of $\bs Q=(Q_1,\ldots,Q_{d+e-1})$ with
respect to $\bs G,\bs H,\bs A,\bs B$ at $(g, h,\bs\alpha,\bs\beta)$.
Since $\pi$ is the linear mapping defined by the projection to
$\cF^{d+e-1}\times \cF^{d+e-1}$, it is
unramified at $(g,h,\bs\alpha,\bs\beta)$ if and only if the Jacobian
matrix $((\partial \bs Q)/(\partial(\bs G,\bs
H))(g,h,\bs\alpha,\bs\beta)$ is of full rank.

Our first result considers the special case $(\bs\alpha,\bs 0)$.
\begin{lemma}
\label{lemma:unramifiedalpha0} For an $\bs\alpha$ with $R^{(2)} (\bs
\alpha) \neq 0$, $\pi$ is unramified at any point $(g,
h,\bs\alpha,\bs 0)\in\pi^{-1}(\bs\alpha,\bs 0)$.
\end{lemma}
\begin{proof}
We let $q=d+e-1$ and consider the $q \times q$ Jacobian matrix
\linebreak $J=(\partial \bs Q/(\partial(\bs G,\bs
H))(g,h,\bs\alpha,\bs 0)$:
\begin{equation}
\label{Jacobian dQ/dG,H}
J= \left(\begin{array}{ccccccc} h(\alpha_1)^{d-1} &
\cdots & h(\alpha_1) & \!\!1\!\! &
g'(h(\alpha_1))\,\alpha_1^{e-1}&\cdots &
g'(h(\alpha_1))\,\alpha_1\\
\vdots & &\vdots &\!\!\vdots\!\! & \vdots & & \vdots\\
h(\alpha_d)^{d-1} &\cdots & h(\alpha_d) & \!\!1\!\! &
g'(h(\alpha_d))\,\alpha_d^{e-1}&\cdots &
g'(h(\alpha_d))\,\alpha_d\\[1ex]
h(\alpha_{d+1})^{d-1} &\cdots & h(\alpha_{d+1}) & \!\!1\!\! &
g'(h(\alpha_{d+1}))\,\alpha_{d+1}^{e-1}\!\!&\cdots &
\!\!g'(h(\alpha_{d+1}))\,\alpha_{d+1}\\
\vdots & &\vdots &\!\!\vdots\!\! & \vdots & & \vdots\\
h(\alpha_{q})^{d-1} & \cdots & h(\alpha_{q}) & \!\!1\!\! &
g'(h(\alpha_{q}))\,\alpha_{q}^{e-1}&\cdots &
g'(h(\alpha_{q}))\,\alpha_{q}
\end{array}\right).
\end{equation}
From Corollary \ref{coro:zero_dimensional_fiber_pi} we have $\#
h(\bs\alpha)=d$ and by Lemma \ref{lemma:cardinality2}
that $h(\alpha_i)\not= h(\alpha_j)$ for $1\le i<j\le d$, while each
of the remaining values $h(\alpha_{d+j})$ agrees with an
$h(\alpha_{i_j})$ with $1\le i_j\le d$ for $1\le j\le e-1$. As a
consequence, the upper-left $d\times d$ submatrix $J_1$ of $J$ is
invertible. Further, by subtracting the $i_j$th row from the
$(d+j)$th row of $J$ for $1\le j\le e-1$, we obtain a $q\times q$
matrix with the same rank as $J$ and of the following shape:
$$ \left(\begin{array}{c|c}J_1 &  * \\\hline
\bs 0 &\begin{array}{ccc} g'(h(\alpha_{d+1}))\,(\alpha_{d+1}^{e-1}-
\alpha_{i_1}^{e-1})&\cdots & g'(h(\alpha_{d+1}))\,(\alpha_{d+1}-
\alpha_{i_1})\\[1ex]
\vdots & & \vdots\\
g'(h(\alpha_{q}))(\alpha_{q}^{e-1}-\alpha_{i_{e-1}}^{e-1})&\cdots
& g'(h(\alpha_{q}))(\alpha_{q}-\alpha_{i_{e-1}})\end{array}
\end{array}\right).$$
It follows that $J$ is invertible if and only if this $(e-1)\times
(e-1)$ lower-right submatrix $J_2$ is invertible.

Since $g$ is monic polynomial of degree $d$ vanishing on
$h(\bs\alpha)$ and $\# h(\bs\alpha)=d$, any $h(\alpha_i)$ is a
simple root of $g$. This implies that $g'(h(\alpha_{d+i}))\not=0$
for $1\le i\le e-1$. We conclude that $J_2$ is invertible if and
only if the matrix
$$J_3=\left(\begin{array}{ccc} \alpha_{d+1}^{e-1}-
\alpha_{i_1}^{e-1}&\cdots & \alpha_{d+1}-
\alpha_{i_1}\\[1ex]
\vdots & & \vdots\\
\alpha_{q}^{e-1}-\alpha_{i_{e-1}}^{e-1}&\cdots &
\alpha_{q}-\alpha_{i_{e-1}}\end{array}\right)$$
is nonsingular. Up to row and column permutations and signs,
 $J_3$ is the matrix in \eqref{matrixEquation2}.
(In the numbering of \eqref{matrixEquation2}, $d+j$ is in the same block
of the partition as $i_j$.)
Since $\poly^{(2)}(\bs\alpha) \neq 0$, it follows that
$J_3$ is nonsingular.
\end{proof}

In fact, the lemma's statement and proof are valid for any $\bs \beta$ with $\# \bs \beta \leq d$,
not just $\bs \beta = \bs 0$,
using Proposition \ref{interpolForSmallSizeBeta} instead of Corollary \ref{coro:zero_dimensional_fiber_pi}.

Next we show that the field extension $\F(\bs A,\bs
B)\hookrightarrow\F(\Gamma_{n,d})$ is separable.
This will show that $\#\pi^{-1}(\bs\alpha,\bs\beta)=\deg\pi$ for
any suitably generic $(\bs\alpha,\bs\beta)$.
\begin{proposition}\label{prop:FieldExtensionFGammaIsSeparable}
The field extension $\F(\bs A,\bs B)\hookrightarrow\F(\Gamma_{n,d})$
is separable.
\end{proposition}
\begin{proof}
The Jacobian $\bs J=\partial \bs
Q/\partial(\bs G,\bs H) \in \F[\bs G,\bs H, \bs A]^{q \times q}$
is the matrix in \eqref{Jacobian dQ/dG,H}
 with $g, h, \alpha$
replaced by vectors of indeterminates $\bs G,\bs H, \bs A$.
 Lemma \ref{lemma:unramifiedalpha0} shows that $\det \bs J$ does not vanish at certain points,
so in particular $\det \bs J \neq 0$ as a polynomial.

In the proof of Lemma \ref{lemma:geometryGamma_n,d} it is shown that
the $\F$-algebra morphism $\overline{\Phi} \colon \F[\Gamma_{n,d}]\to\F[\bs
A,\bs G,\bs H]$, defined by $\overline{\Phi}(\overline{B}_j)=(G\circ
H)(A_j)$ for $1\le j\le d+e-1$, where $\overline{B}_j\in
\F[\Gamma_{n,d}]=\F[\bs A,\bs B,\bs G,\bs
H]/\mathcal{I}(\Gamma_{n,d})$ is the coordinate function of
$\Gamma_{n,d}$ defined by $B_j$, is an isomorphism. As $\det
\bs J$ is a nonzero element of $\F[\bs A,\bs G,\bs H]$,
we have $\overline{\Phi}^{-1}(\det
{\bs J})\not=0$, and the fact that $\F[\Gamma_{n,d}]$ is a
domain implies that $\overline{\Phi}^{-1}(\det {\bs J})$ is
not a zero divisor of $\F[\Gamma_{n,d}]$. As a consequence,
$\Gamma_{n,d}\cap \{\det {\bs J}\not=0\}$ is a nonempty
Zariski open subset of $\Gamma_{n,d}$. In particular, the projection
$\pi:\Gamma_{n,d}\to\cF^{d+e-1}\times\cF^{d+e-1}$ is unramified at
any $(\bs g,\bs h,\bs \alpha,\bs\beta)$ in this set.

According to Iversen \cite[Proposition T.8]{Iversen73}, the extension
$\F(\bs A,\bs B)\hookrightarrow\F(\Gamma_{n,d})$ is separable if and
only if $\pi \colon \Gamma_{n,d}\to\cF^{d+e-1}\times\cF^{d+e-1}$ is
unramified in a nonempty Zariski open subset of $\Gamma_{n,d}$,
which proves the proposition.
\end{proof}

Now \cite[Proposition 1]{Heintz83}
shows that $\deg\pi=\max\#\pi^{-1}(\bs \alpha,\bs\beta)$, where the
maximum is taken among all fibers 
with finite cardinality, and there exists a nonempty Zariski open
subset of $\cF^{d+e-1}\times \cF^{d+e-1}$ such that
$\#\pi^{-1}(\bs \alpha,\bs\beta)=\deg\pi$ for all its elements.
Schost \cite[Proposition 3]{Schost03} and its proof allow to describe more precisely
an open subset for the condition that
$\#\pi^{-1}(\bs \alpha,\bs\beta)=\deg\pi$.
\begin{fact}\label{coro:BoundGenericityConditionDegPiPoints}
There exists a nonzero polynomial $R^{(3)}\in\F[\bs A,\bs B]$ of total
degree at most $2\deg\pi \deg\Gamma_{n,d}$ with the following
property: for any $(\bs \alpha,\bs\beta)\in \cF^{d+e-1}\times
\cF^{d+e-1}$ with $R^{(3)}(\bs \alpha,\bs\beta)\not=0$, we have
$\#\pi^{-1}(\bs \alpha,\bs\beta)=\deg\pi$.
\end{fact}

We have introduced two polynomial conditions on $\bs \alpha$: $R^{(2)}$ in \eqref{Qdef} and
$R^{(3)}$ in the previous fact.
We now combine them into a single polynomial:
\begin{equation}
\label{defR}
R = R^{(2)} R^{(3)}.
\end{equation}
Then $R \in \F[\bs A]$ is nonzero and from
\eqref{StirlingBound},
\eqref{piBounds}, and  Lemma \ref{lemma:geometryGamma_n,d}, we have
\begin{equation}
\label{degR}
\deg R <  (d+e)^{d+e-1} + 2 (d^2e)^{d+e-1} \leq 3 (d^2e)^{d+e-1}.
\end{equation}

\begin{definition}
\label{defGeneric} For $d,e \geq 2$ and a field $\F$, we call $\bs \alpha \in
\F^{d+e-1}$ \emph{generic} if $R(\bs \alpha) \neq 0$, and the
nonempty Zariski-open set $\mathcal U = \{R\neq 0\} \subset
\F^{d+e-1}$ consists of these generic $\bs \alpha$. For any $\bs
\beta \in \F^{d+e-1}$, we also call $(\bs \alpha, \bs \beta)$
\emph{generic}.
\end{definition}

We can now give a criterion under which $\pi$ is unramified.
\begin{lemma}\label{lemma:FiberUnramified}
Let $(\bs\alpha,\bs \beta)\in\cF^{d+e-1}\times \cF^{d+e-1}$ be generic.
Then $\#\pi^{-1}(\bs \alpha,\bs\beta)=\deg\pi$ and $\pi$ is
unramified at all $(g,h,\bs\alpha,\bs
\beta)\in\pi^{-1}(\bs\alpha,\bs \beta)$.
\end{lemma}
\begin{proof}
From Fact \ref{coro:BoundGenericityConditionDegPiPoints},
we know that $\#\pi^{-1}(\bs\alpha,\bs\beta)=\deg\pi$.
Thus the restriction
$\pi|_{\pi^{-1}(\mathcal U)} \colon \pi^{-1}(\mathcal U)\to \mathcal U$ is
a finite morphism. Then for any $(\bs\alpha,\bs\beta)\in \mathcal U$,
the Principle of conservation
of number (see, e.g., Danilov \cite[\S 5.7]{Danilov94}) asserts that
$$\deg\pi=\sum_{(g,
h,\bs\alpha,\bs\beta)\in\pi^{-1}(\bs\alpha,\bs\beta)}\deg_{(g,
h,\bs\alpha,\bs\beta)}\pi,$$
where
$$\deg_{(g,
h,\bs\alpha,\bs\beta)}\pi=\dim_\F(\mathcal{O}_{\Gamma_{n,d},(g,
h,\bs\alpha,\bs\beta)}/\pi^*(\mathfrak{m}_{(\bs\alpha,\bs\beta)})
\mathcal{O}_{\Gamma_{n,d},(g,h,\bs\alpha,\bs\beta)}),$$
$\mathcal{O}_{\Gamma_{n,d},(g,h,\bs\alpha,\bs\beta)}$ is the local
ring of rational functions on $\Gamma_{n,d}$ defined over $\F$ which
are well-defined at $(g,h,\bs\alpha,\bs\beta)$ and
$\mathfrak{m}_{(\bs\alpha,\bs\beta)}$ is the maximal ideal of
rational functions on $\cF^{d+e-1}\times \cF^{d+e-1}$, defined over
$\F$ and vanishing at $(\bs\alpha,\bs\beta)$.

Now the fact that $\#\pi^{-1}(\bs\alpha,\bs\beta)=\deg\pi$ implies
that $\deg_{(g,h,\bs\alpha,\bs\beta)}\pi=1$ for any $(g,
h,\bs\alpha,\bs\beta)\in\pi^{-1}(\bs\alpha,\bs\beta)$. This means
that $\mathcal{O}_{\Gamma_{n,d},(g,
h,\bs\alpha,\bs\beta)}/\pi^*(\mathfrak{m}_{(\bs\alpha,\bs\beta)})
\mathcal{O}_{\Gamma_{n,d},(g, h,\bs\alpha,\bs\beta)}$ is isomorphic
to $\F$. It follows that
$$\pi^*(\mathfrak{m}_{(\bs\alpha,\bs\beta)})=
\mathfrak{m}_{(g,h,\bs\alpha,\bs\beta)}$$
for all the points $(g,
h,\bs\alpha,\bs\beta)\in\pi^{-1}(\bs\alpha,\bs\beta)$, which by,
e.g., \cite[\S 5.2]{Danilov94}, implies that $\pi$ is unramified at
all such points.
\end{proof}

%
%
\section{Homotopy for interpolation}
Our general task is to find composite interpolating polynomials
for given generic $(\bs\alpha,\bs\beta)\in \F^{d+e-1} \times \F^{d+e-1}$,
where $\char(\F)$ does not divide $d$.
Corollary \ref{coro:zero_dimensional_fiber_pi} provides a solution for the special value $\bs \beta = \bs 0$.
Our main tool is a homotopy connecting these two instances.

Recall that the interpolants for
$(\bs\alpha,\bs\beta)=(\alpha_1,\ldots,\alpha_{d+e-1},\beta_1,
\ldots,\beta_{d+e-1})$ are the common zeros $(g,h)$ of the polynomial system
\begin{equation}\label{eq:PolynomialSystemInterpolant}
(G\circ H)(\alpha_i)-\beta_i=0\quad\text{ for } 1\le i\le d+e-1 ,
\end{equation}
where $G=X^d+G_{d-1}X^{d-1}+\cdots+G_0$ and
$H=X^e+H_{e-1}X^{e-1}+\cdots +H_1X$ are univariate polynomials in
$X$ with $d+e-1$ indeterminate coefficients $G_{d-1},\ldots,G_0$ and
$H_{e-1},\ldots,H_1$. The deformation describing the homotopy is the
variety $V^*\subset\cF^{d+e}$, defined over $\F$ and determined by the
equations
\begin{equation}
\label{Vequations}
(G\circ H)(\alpha_i)-\beta_i S=0\quad\text{ for } 1\le i\le d+e-1 .
\end{equation}
Here, $S$ is a parameter whose values range over $\cF$. For $S=0$, we have
the special problem already solved, and $S=1$ is our general task.
More precisely, we shall be interested in a curve $V\subset V^*$
which consists of certain irreducible components over $\F$ of the
variety $V^*$.
Since each equation in \eqref{Vequations} has degree at most $d$,
the B\'ezout inequality then implies that
\begin{equation}
\label{degVupperBound}
\delta = \deg V \leq \deg V^* \leq d^{d+e-1}.
\end{equation}

The algorithm solving system \eqref{eq:PolynomialSystemInterpolant}
may be divided into three main parts. First, we  compute a bivariate
polynomial $m_i(S,T)$ defining an irreducible plane curve
birationally equivalent to an irreducible component of $V$.
Then we extend the computation of the polynomial $m_i(S,T)$ to the
computation of the birational inverse itself. Finally, 
we substitute 1 for $S$ in the birational inverse computed in the
previous step and recover at least one solution to system
\eqref{eq:PolynomialSystemInterpolant}.
%
%
We use a representation of varieties called a {\em geometric
solution} or a {\em rational univariate representation} which is well suited for
algorithmic purposes (see, e.g., Rouillier \cite{Rouillier97},
Giusti et al.\ \cite{GiLeSa01}).

\subsection{Geometric solutions}\label{subsec:GeometricSolutions}
Let $X_1,\ldots,X_n$ be indeterminates over $\F$ and denote $\bs
X=(X_1,\ldots,X_n)$. We start with the definition of a geometric
solution of a zero-dimensional variety defined over $\F$. Let $V=
\{\bs P_1, \ldots,\bs P_D \}  \subset \cF^n$ be a zero-dimensional
reduced variety defined over $\F$, and $\mathcal{L}\in \F[\bs X]$ a
linear form  which separates the points of $V$, that is,
$\mathcal{L} (\bs P_i)\ne \mathcal{L} (\bs P_j)$ for $i \ne j$. A
\emph{geometric solution} of $V$ consists of
\begin{itemize}
\item $\bs \lambda \in\F^n$ so that the linear form $\mathcal{L}=\bs \lambda\cdot
\bs X=\lambda_1X_1+\cdots+\lambda_n X_n\in \F[\bs X]$
separates the points of $V$,
\item the minimal polynomial  $m_{\bs\lambda}= \prod_{1 \le i \le D}
(T - \mathcal{L}(\bs P_i))\in \F[T]$ of $\mathcal{L}$ in $V$, where
$T$ is a new variable,
\item polynomials $w_1, \dots,w_n \in \F[T]$ with $\deg w_j< D$ for
$1\le j \le n$ such that
$$ V = \{(w_1(\eta) ,\dots,w_n (\eta) ) \in
\cF^n \colon \,\eta \in \cF,\,
 m_{\bs\lambda}(\eta) = 0 \}.$$
\end{itemize}
Thus $\bs P_i=\big(w_1(\mathcal L(\bs P_i)),\ldots, w_n(\mathcal L(\bs
P_i)\big)$ for $1\le i \leq D$, and each $w_j$ interpolates
on $D$ pairwise distinct points $\mathcal L(\bs P_1),\ldots, \mathcal L(\bs P_D)$.

Next, we consider this notion for equidimensional varieties of
dimension one, that is, for curves, defined over $\F$. Let $n\geq
2$, $V\subset \cF^n$ be a reduced variety defined over $\F$ of pure
dimension $1$ and degree $\delta$, and let $V=V_1\cup\cdots\cup V_t$
be the decomposition over $\F$ of $V$ into irreducible components.
Suppose that the indeterminate $X_1$ forms a separable transcendence
basis of all field extensions $\F \hookrightarrow \F(V_i)$ for $1\le
i\le t$, that is, each $\F(X_1) \hookrightarrow \F(V_i)$ is a finite
separable field extension. Denote by $D_i$ its degree and let
$D=\sum_{1\leq i \leq t} D_i$. In particular, the linear projection
$\pi:V_i\to \cF^1$ defined by $\pi(\bs x)=x_1$ is a dominant
morphism of degree $D_i$ for $1\le i\le t$. The behavior of the
degree of a variety under linear maps (see, e.g., \cite[Lemma
2]{Heintz83}) implies that
\begin{equation}
\label{Dbound}
D=\sum_{1\leq i \leq t}D_i\le\sum_{1 \leq i \leq t} \deg V_i= \delta.
\end{equation}

Let  $\bs\Lambda=(\Lambda_2,\ldots, \Lambda_{n})$ be  a vector of
new indeterminates. The extension $V_{\F(\bs\Lambda)}$ of $V$ to
$\overline{\F(\bs\Lambda)}{}^n$ is a variety of pure dimension $1$
and its coordinate ring is isomorphic to
$\F(\bs\Lambda)\otimes_{\F}\F[V]$. Consider the generic linear form
\begin{equation}\label{eq: forma lineal generica}
\mathcal{L}_{\bs\Lambda}=\bs\Lambda\cdot\bs X^*=\Lambda_2X_2+\cdots+
\Lambda_nX_n.\end{equation}
As the mapping $\F[V]\to\F[V_1]\times\cdots\times\F[V_t]$ defined by
$\xi\mapsto (\xi|_{V_1},\ldots,\xi|_{V_t})$ is an isomorphism of
$\F$-algebras, we identify $\F[V]$ with
$\F[V_1]\times\cdots\times\F[V_t]$. For $\xi\in\F[V]$, we denote its
restriction to $V_i$ by $\xi^{(i)}$ for $1\le i\le t$. Let
$\xi_1,\ldots, \xi_n\in \F[V]$ be the coordinate functions of $V$
induced by $X_1,\ldots,X_n$, and set $\widehat{\mathcal{L}}=
\mathcal{L}_{\bs\Lambda}(\xi_2,\ldots,\xi_n)\in
\F(\bs\Lambda)\otimes_{\F}\F[V]$. Since $\xi_1^{(i)}$ and
$\widehat{\mathcal{L}}^{(i)}$ are algebraically dependent over
$\F(\bs\Lambda)$, there exists an irreducible polynomial
$m_{\bs\Lambda}^{(i)}\in \F[\bs\Lambda,X_1,T]$, separable with
respect to $T$, such that
$m_{\bs\Lambda}^{(i)}(\bs\Lambda,\xi_1^{(i)},\widehat{\mathcal{L}}^{(i)})=0$
holds in $\F(\bs\Lambda)\otimes_{\F}\F[V_i]$ for $1\le i\le t$.
These polynomials are pairwise relatively prime, and defining
$$m_{\bs\Lambda}=\prod_{1\leq i \leq t} m_{\bs\Lambda}^{(i)},$$
we find the following identity in
$\F(\bs\Lambda)\otimes_{\F}\F[V]$:
\begin{equation} \label{eq: polinomio de Chow}
    m_{\bs\Lambda}(\bs\Lambda,\xi_1,\widehat{\mathcal{L}})=0.
\end{equation}
From, e.g., \cite[Proposition 1]{Schost03}, we deduce the following
bounds:
\begin{itemize}
    \item $\deg_T m_{\bs\Lambda}= D$,
    \item $\deg_{X_1} m_{\bs\Lambda} \leq \delta$,
    \item $\deg_{\bs\Lambda} m_{\bs\Lambda} \leq \delta$.
\end{itemize}

Taking partial derivatives with respect to $\Lambda_j$ in \eqref{eq:
polinomio de Chow}, we deduce the identity
\begin{equation}\label{eq: primera derivada de Chow}
    \displaystyle \frac{\partial m_{\bs\Lambda}}{\partial \Lambda_{j}}
    (\bs\Lambda,\xi_1, \widehat{\mathcal{L}}) + \xi_j
    \displaystyle\frac{\partial m_{\bs\Lambda}}{\partial T}
    (\bs\Lambda,\xi_1, \widehat{\mathcal{L}})=0
\end{equation}
in  $\F(\bs\Lambda)\otimes_{\F}\F[V]$ for $2\le j\le n$. The
polynomial $\partial m_{\bs\Lambda}/\partial T$ is nonzero, and the
resultant $\mathrm{Res}_T(m_{\bs\Lambda},\partial
m_{\bs\Lambda}/\partial T)$ is also nonzero (see, e.g., G\'\i menez \&\ Matera \cite[Lemma
3.2]{GiMa19}).

We take some $\bs\lambda\in \F^{n-1}$ such that $\deg_T
m_{\bs\Lambda}(\bs\lambda,X_1,T)=D$ and \linebreak
$\mathrm{Res}_T(m_{\bs\Lambda},\partial m_{\bs\Lambda}/\partial
T)(\bs\lambda,X_1)\not=0$. Then the linear form
$\mathcal{L}=\bs\lambda\cdot \bs X^*$ with $\bs X^* = (X_2,\ldots,
X_n)$ induces a primitive element of the separable field extension
$\F(X_1)\hookrightarrow \F(V_i)$ for $1\le i\le t$, and the minimal
polynomials of $\mathcal{L}$ in the field extensions
$\F(X_1)\hookrightarrow \F(V_i)$ are pairwise relatively prime (see,
e.g., \cite[Proposition 3.4]{GiMa19}). If $\ell$ is the coordinate
function of $\F[V]$ defined by $\mathcal{L}$, we then say that
$\ell$ is a {\em primitive element} of the extension
$\F(X_1)\hookrightarrow \F(V)$, where $\F(V)$ is total ring of
fractions of $\F[V]$. Since $\deg_T m_{\bs\Lambda} = D$, it follows
that $m_{\bs\Lambda}(\bs\lambda,X_1,T)$ is the minimal polynomial of
$\ell$ in the extension $\F(X_1)\hookrightarrow \F(V)$.
Setting
\begin{equation}\label{eq:definitionParametrizChow}
m_{\bs\lambda}(X_1,T) = m_{\bs\Lambda}(\bs\lambda,X_1,T), \quad v_j
(X_1,T) = -\frac{\partial m_{\bs\Lambda}}{\partial \Lambda_{j}}
(\bs\lambda,X_1,T) \text{ in } \F[X_1,T]
\end{equation}
and substituting $\bs\lambda$ for $\bs\Lambda$ in (\ref{eq:
polinomio de Chow})--(\ref{eq: primera derivada de Chow}), we obtain
the following identities of $\F[V]$:
\begin{equation}\label{eq:parametrizacionesChow}
m_{\bs\lambda}(\xi_1, \ell) = 0, \quad \frac{\partial
m_{\bs\lambda}}{\partial T} (\xi_1, \ell)\xi_j- v_j (\xi_1,\ell)=0\
\ (2\le j\le n).
\end{equation}
Thus the polynomials
$$m_{\bs\lambda}(X_1,\mathcal{L}) \text{ and }
\frac{\partial m_{\bs\lambda}}{\partial T}(X_1,\mathcal{L})X_j-
v_j(X_1,\mathcal{L})\ \ (2\le j\le n)$$
belong to $\mathcal{I}(V) \subset \F[X_1, \ldots, X_n]$.
They satisfy the bounds $\deg_{X_1} v_j \leq \delta$
and $\deg_T v_j \leq D$ for $2\le j\le n$.
Finally, the equations
$$m_{\bs\lambda}(X_1,\mathcal{L})=0,\quad \frac {\partial
m_{\bs\lambda}} {\partial T}(X_1,\mathcal{L})X_j-
v_j(X_1,\mathcal{L})=0\quad (2\le j\le n),$$
constitute a system of equations for the variety $V$ in the Zariski
dense open subset $V\cap \{(\partial m_{\bs\lambda}/\partial
T)(X_1,\mathcal{L})\not=0\}$ of $V$. This motivates the following
definition.
\begin{definition}\label{def:GeometricSolution}
With assumptions as above, a \emph{geometric solution} of $V$
consists of the following items:
    \begin{itemize}
    \item a linear form $\mathcal{L}=\bs\lambda\cdot \bs X^*\in \F[\bs X^*]$
    which induces a primitive element $\ell$ of the extension
    $\F(X_1)\hookrightarrow \F(V)$,
    \item the minimal polynomial $m_{\bs\lambda}\in \F[X_1][T]$ of $\ell$,
    \item a generic \emph{parametrization} $v_2, \ldots, v_n$ of $V$ by the
    zeros of $m_{\bs\lambda}$, of the form
    $$\frac {\partial m_{\bs\lambda}}
   { \partial T}(X_1,T)\, X_j-v_j(X_1,T)\quad
    (2\le j\le n),$$ with $v_j\in \F[X_1][T]$,
    $\deg_Tv_j<D$, $\deg_{X_1}v_j\le\delta$ and $(\partial m_{\bs\lambda}
    /\partial T)(X_1,\mathcal{L})\,X_j-v_j(X_1,\mathcal{L})\in \mathcal{I}(V)$ for $j \leq n$.
    \end{itemize}
\end{definition}
We observe that the polynomial $m_{\bs\lambda}\in \F[X_1,T]$ can be
also defined as follows: consider the linear map
$$\pi_{\bs\lambda} \colon V\to \cF^2,\quad\pi_{\bs\lambda}(\bs
x)=(x_1,\bs\lambda\cdot \bs x^*).$$
The Zariski closure of $\pi_{\bs\lambda}(V)$ is a plane curve of
degree at most $\delta$, which is indeed defined by
$m_{\bs\lambda}(X_1,T)=0$. Further, the projection
$\pi_{\bs\lambda}$ constitutes a birational equivalence between $V$
and the plane curve $\mathcal{C}=\{m_{\bs\lambda}(X_1,T)=0\}$,
namely there is a rational inverse of $\pi_\lambda$. More precisely,
let
\begin{align*}
V_{\mathrm{deg}}&=\{\bs x\in\cF^n:(\partial m_{\bs\lambda}/\partial
T)(x_1,
\mathcal{L}(\bs x))=0\},\\
\mathcal{C}_{\mathrm{deg}}&=\{\bs (x_1,t)\in\cF^2:(\partial
m_{\bs\lambda}/\partial T)(x_1,t)=0\}.
\end{align*}
The fact that $\mathrm{Res}_T(m_{\bs\lambda},\partial
m_{\bs\lambda}/\partial T)\not=0$ implies that $V\setminus
V_{\mathrm{deg}}$ and
$\mathcal{C}\setminus\mathcal{C}_{\mathrm{deg}}$ are Zariski dense
open subsets of $V$ and $\mathcal{C}$ respectively. We have the
following result.
\begin{fact}[{Cafure \&\ Matera \cite[Proposition
6.3]{CaMa06}}]\label{fact:PlaneCurveBirrational}
$\pi_{\bs\lambda}|_{V\setminus V_{\mathrm{deg}}}:V\setminus
V_{\mathrm{deg}}\to\mathcal{C}\setminus\mathcal{C}_{\mathrm{deg}}$
is an isomorphism of open sets defined over $\F$.
\end{fact}

We remark that the inverse of $\pi_{\bs\lambda}|_{V\setminus
V_{\mathrm{deg}}}$ is defined in the following way:
\begin{align*}
\mathcal{C}\setminus\mathcal{C}_{\mathrm{deg}}&\to V\setminus V_{\mathrm{deg}},\\
(\bs x,t) & \mapsto \left(x_1,\frac{v_2(x_1,t)}{(\partial
m_{\bs\lambda}/\partial
T)(x_1,t)},\ldots,\frac{v_n(x_1,t)}{(\partial
m_{\bs\lambda}/\partial T)(x_1,t)}\right).
\end{align*}
%
%
\subsection{Power series
expansions}\label{subsec:PowerSeriesExpansions}
Let $n\geq 2$, $V\subset \cF^n$ be a reduced variety defined over
$\F$ of pure dimension $1$ and degree $\delta$, and let
$V=V_1\cup\cdots\cup V_t$ be the decomposition over $\F$ of $V$ into
irreducible components. Let $\mathcal{L}=\bs\lambda\cdot \bs X^*\in
\F[\bs X^*]$ be a linear form and $m_{\bs\lambda},v_2, \ldots,
v_n\in \F[X_1][T]$ which form a geometric solution of $V$ as in the
previous subsection.

$V$ is birationally equivalent to the plane curve
$\mathcal{C}\subset \cF^2$ defined by $m_{\bs\lambda}$. The fact
that $V=V_1\cup\cdots\cup V_t$, where each $V_i$ is irreducible of
dimension 1, implies $\mathcal{C}$ has a decomposition
$\mathcal{C}=\mathcal{C}_1\cup\cdots\cup \mathcal{C}_t$, where each
$\mathcal{C}_i$ is defined by an irreducible factor $m_i$ of
$m_{\bs\lambda}$, so that the irreducible factorization of
$m_{\bs\lambda}$ in $\F[X_1,T]$ is of the form
$$m_{\bs\lambda}(X_1,T)=\prod_{i=1}^tm_i(X_1,T).$$
As the linear form $\mathcal{L}$ is a primitive element of the
extension $\F(X_1)\hookrightarrow\F(V)$, the factors $m_i$ are
pairwise coprime with $\sum_{i=1}^t\deg_Tm_i=\sum_{i=1}^tD_i=D$.

Consider a single factor, say $m_1$, of $m_{\bs\lambda}$. A
representation of the zeros of $m_1$ is provided by the well-known
Hensel lemma (see, e.g., Abhyankar \cite[Lecture 12]{Abhyankar90}). More
precisely, assume that
\begin{equation}\label{eq:ConditionDiscriminant}
\mathrm{Res}_T \left(m_1(0,T),\frac{\partial m_1}{\partial
T}(0,T)\right)\not=0.
\end{equation}
Then there is a factorization
$m_1=\prod_{i=1}^{D_1}(T-\sigma_i)$ in $\cF[\![X_1]\!][T],$
where $\cF[\![X_1]\!]$ denotes the ring of formal power series in
$X_1$ with coefficients in $\cF$, and $\sigma_i\not=\sigma_j$ for
$i\not=j$. We shall see that an approximation of order $2D\delta$ of
a single power series, say $\sigma=\sigma_1$, suffices to compute
$m_1$.

For this purpose, we observe that $m_1\in\F[X_1,T]$ can also be
characterized as the primitive polynomial $p\in\F[X_1][T]$ of
minimal degree (up to nonzero multiples in $\F$) for which
$p(X_1,\sigma)=0$ holds. Let $N=2D\delta$ and let $\sigma_N\in
\F[X_1]$ be the power series $\sigma$ truncated up to order $N+1$,
that is, $\sigma_N$ is the polynomial of degree at most $N$
congruent to $\sigma$ modulo $X_1^{N+1}$. Our next result shows that
$m_1$ can be obtained as the solution of a suitable congruence
equation involving $\sigma_N$.
\begin{lemma}\label{lemma:ComputationMinimal} Let $p\in \F[X_1,T]$ be
a polynomial with $\deg_{X_1}p\le \delta$ and $\deg_Tp\le D$
satisfying the condition
\begin{equation}\label{eq:EquivRelationMinimal} p(X_1,\sigma_N)
\equiv 0\mod X_1^{N+1}.\end{equation}
Then $m_1$ divides  $p$ in $\F[X_1,T]$.
\end{lemma}
\begin{proof} Let $p\in \F[X_1,T]$ be a solution of
\eqref{eq:EquivRelationMinimal} satisfying the conditions of the
lemma. The resultant $q\in \F[X_1]$ of $p$ and $m_1$ with respect to
$T$ has degree at most $N$ and belongs to the ideal generated  by
$m_1$ and $p$. Since $m_1(X_1,\sigma_N)\equiv 0\mod X_1^{N+1}$ and
$p(X_1,\sigma_N)\equiv 0\mod X_1^{N+1}$ by hypothesis, we deduce
that $q(X_1)\equiv 0\mod X_1^{N+1}$. Therefore, the facts that $\deg
q\le N$ and $q(X_1)\equiv 0\mod X_1^{N+1}$ imply $q=0$. In
particular, $m_1$ and $q$ have a nonconstant common factor in
$\F(X_1)[T]$. Taking into account the irreducibility of $m_1$ in
$\F(X_1)[T]$ and the Gauss lemma, we easily deduce the statement of
the lemma.
\end{proof}

Lemma \ref{lemma:ComputationMinimal} characterizes $m_1$ as the
nonzero solution of \eqref{eq:EquivRelationMinimal} of minimal
degree in $T$, up to nonzero multiples in $\F$. We remark that the
argument of Lemma \ref{lemma:ComputationMinimal} holds {\em mutatis
mutandis} with any $M=2D'\delta'$ instead of $N$, where $D'$ and
$\delta'$ are upper bounds of $D$ and $\delta$ respectively.
%
%
\section{The homotopy curve}
This section studies the set $V^*\subset\cF^{d+e}$ of common
solutions of the equations \eqref{Vequations}
 and provides a homotopy between the solutions for $(\bs \alpha,\bs\beta)$
 and for $(\bs \alpha,\bs 0)$.
 We use throughout the notation given there and assume that $\bs \alpha$ is generic.

 We consider the projection $\pi_S \colon V^*\to\cF^1$ with $\pi_S(\bs g,\bs h, s)=s$.
\begin{lemma}\label{Lemma:Pi_S_Dominant}
$\pi_S$ is dominant.
\end{lemma}
\begin{proof}
For $s\in \cF$, we write $\pi_S^{-1}(s)= W_s \times \{s\}$, where
$W_s \subset\cF^{d+e-1}$ is the set of solutions $(g,h)$ of the system
$(G\circ H)(\alpha_i)-s\beta_i=0$ $\text{ for } 1\le i\le d+e-1 $, namely the
set of interpolants for the interpolation problem determined by
$(\bs\alpha,s\bs \beta)$. By Corollary
\ref{coro:ExistenceGenericInterpolant},
any $(\bs\gamma,\bs \delta)\in \mathcal U$ admits a nonempty finite set of
interpolants. As $(\bs\alpha,\bs\beta)$ belongs to this dense Zariski open subset
$\mathcal U \subset \cF^{2(d+e-1)}$, it follows
that the affine line
$L_{(\bs\alpha,\bs\beta)}=\{(\bs\alpha,s\bs\beta) \colon s\in\cF^1\}$ has a
nonempty intersection with $\mathcal U$. In particular,
$L_{(\bs\alpha,\bs\beta)}\cap \mathcal U\}$ is a nonempty Zariski open subset
of $L_{(\bs\alpha,\bs\beta)}$, which proves that there is a nonempty
Zariski open subset $U_S = \{s \colon (\bs \alpha, \bs \beta s) \in L(\bs \alpha, \bs \beta) \cap \mathcal U \}$ of $\cF^1$,
so that the interpolation problem
determined by $(\bs\alpha,s\bs\beta)$ for $s \in U_S$ admits a finite positive
number of interpolants. In particular, $\pi_S^{-1}(s)$ is nonempty
for any $s\in U_S$, which proves that $\pi_S$ is dominant.
\end{proof}

Let $V^*=V_1\cup\cdots\cup V_u$ be the decomposition of $V^*$ into
irreducible components defined over $\F$. Since $\pi_S$ is dominant,
we may assume that the restriction $\pi_S|_{V_i} \colon V_i\to\cF^1$ of
$\pi_S$ to $V_i$ is dominant for $1\le i\le t$ and not dominant for
$t+1\le i\le u$, with some $t \in \{1,\ldots, u\}$.
We now define the curve which describes our deformation:
\begin{equation}
\label{defV}
V=V_1\cup\cdots\cup V_t \subset\cF^{d+e}.
\end{equation}
\begin{lemma}\label{lemma:dimV_i=1}
For $1\le i\le t$, $\dim V_i=1$.
\end{lemma}
\begin{proof}
Fix $i$ with $1\le i\le t$. The Theorem on the dimension of fibers
(see, e.g., \cite[\S 6.3, Theorem 7]{Shafarevich94}) shows that
$$\dim V_i-1=\dim(\pi_S|_{V_i})^{-1}(s)$$
for $s \in {\mathcal U}_S$.
The proof of
Lemma \ref{Lemma:Pi_S_Dominant} shows that $\pi_S^{-1}(s)$ has
dimension 0 for generic $s\in\cF^1$, which implies that
$$\dim V_i-1=\dim(\pi_S|_{V_i})^{-1}(s)\le 0.$$
Taking into account that $V\subset\cF^{d+e}$ is defined by the
$d+e-1$ equations \eqref{Vequations}, it follows
that $\dim V_i\ge 1$ for $1\le i\le u$. We conclude that
$\dim V_i=1$.
\end{proof}

Next we give a more intrinsic definition of the curve $V$. Let
$Q_i=(G\circ H)(\alpha_i)-\beta_iS$ $\text{ for } 1\le i\le d+e-1 $, denote $\bs
Q=(Q_1,\ldots,Q_{d+e-1})$ and let $J$ be the Jacobian determinant of
$\bs Q$ with respect to $\bs G$ and $\bs H$, that is,
$$J=\det(\partial \bs Q/\partial(\bs G,\bs H)).$$
Let $\mathcal{I}\subset\F[\bs G,\bs H, S]$ be the ideal generated by
$Q_1,\ldots,Q_{d+e-1}$ and let
$$\mathcal{I}:J^\infty=\{Q\in \F[\bs
G,\bs H, S] \colon \exists\, m\ge 0\textrm{ with }J^mQ\in\mathcal{I}\}$$
be the saturation of $\mathcal{I}$ with respect to $J$.
Its variety $\mathcal{Z}(\mathcal{I}:J^\infty)\subset\F^{d+e}$ is the
Zariski closure of the locally closed set
$\mathcal{Z}(\mathcal{I})\setminus\mathcal{Z}(J)$.
\begin{lemma}\label{lemma:V_asSaturation}
$V=\mathcal{Z}(\mathcal{I}:J^\infty)$.
\end{lemma}
\begin{proof}
For a component $V_i$ with $i>t$ we have $\dim\pi_S(V_i)=0$.
Further, as $V_i$ is irreducible, also $\pi_S(V_i)$ is irreducible,
so that $\pi_S(V_i)=\{s_i\}$ for a specific $s_i\in\cF^1$. The fiber
$\pi_S^{-1}(s_i)$ is the set of common zeros in
$\cF^{d+e-1}$ of the polynomials $Q_j(\bs G,\bs H,s_i)$ $(1\le j\le
d+e-1)$. For a given point $(\bs g,\bs h,s_i)\in\pi_S^{-1}(s_i)\cap
V_i$, if $J(\bs g,\bs h,s_i)$ was nonzero, then $(\bs g,\bs h,s_i)$ would
be an isolated point of $V_i$, contradicting the fact that $\dim
V_i\ge 1$. It follows that $J(\bs g,\bs h,s_i)=0$ for any $(\bs
g,\bs h,s_i)\in V_i$ and $\mathcal{Z}(\mathcal{I}:J^\infty) \subseteq V$.

On the other hand, let  $1\le i\le t$. By Lemma \ref{lemma:dimV_i=1}, $V_i$
has dimension 1 and
$\pi_S|_{V_i} \colon V_i\to\cF^1$ is dominant. Since $V_{t+1}\cup\cdots\cup
V_u$ projects onto the finite set $\{s_{t+1},\ldots,s_u\}$, for
$s\in\cF^1\setminus\{s_{t+1},\ldots,s_u\}$ the fiber $\pi_S^{-1}(s)$
does not intersect $V_{t+1}\cup\cdots\cup V_u$. Further, since
$\pi_S^{-1}(1)=\pi^{-1}(\bs\alpha,\bs\beta) \times \{1\}$ and
$\#\pi^{-1}(\bs\alpha,\bs\beta)=\deg\pi$ by Fact \ref{coro:BoundGenericityConditionDegPiPoints}
and the genericity of $\bs \alpha \in \mathcal U$,
we conclude that
$\#\pi_S^{-1}(s)=\deg\pi$ for generic $s\in\cF^1$. In particular, by
Lemma \ref{lemma:FiberUnramified} we have that $\pi$ is unramified
at any $(\bs g,\bs h,\bs\gamma,\bs\delta)
\in\pi^{-1}(\bs\gamma,\bs\delta)$ with $(\bs\gamma,\bs\delta)\in
\mathcal U$, which implies that $\pi_S$ is unramified at any $(\bs g, \bs
h,s) \in\pi_S^{-1}(s)$ with $s$ in the nonempty Zariski open subset
$U_S$ of $\cF^1$.

We pick some $s\in\pi_S(V_i) \cap U_S$
and $(\bs g,\bs h,s)\in\pi_S^{-1}(s)$.
Then $J(\bs g,\bs h,s)\not=0$, so that
$V_i$ is not contained in $\{J=0\}$. It
follows that the Zariski closure $V_i$ of $V_i\setminus\{J=0\}$ is
contained in the zero set $\mathcal{Z}(\mathcal{I}:J^\infty)$,
which completes the proof of the lemma.
\end{proof}

With a slight abuse of notation, we also denote the
restriction of $\pi_S$ to $V$ as $\pi_S$, that is, we consider
the projection
$$\pi_S \colon V\to\cF^1 \text{ with } \pi_S(\bs g,\bs h,s)=s.$$
As $V=V_1\cup\cdots\cup V_t$ is the irreducible decomposition of $V$
over $\F$ and $\pi_S|_{V_i} \colon V_i\to\cF^1$ is dominant for
$1\le i\le t$, we define
$$\deg\pi_S=\sum_{1 \leq i \leq t} \deg \pi_S|_{V_i}.$$
In the next corollary we collect several consequences of Lemma
\ref{lemma:V_asSaturation}.
\begin{corollary}
\label{ViSeparable}
The following assertions hold:
\begin{enumerate}
  \item the extension $\F(S)\hookrightarrow\F(V_i)$ is separable for $1\le
i\le t$;
\item $\deg\pi_S=\deg\pi$;
\item $\mathcal{I}:J^\infty$ is a radical ideal.
\end{enumerate}
\end{corollary}
\begin{proof}
Fix $i$ with $1\le i\le t$.
Since $V_i$ is not
contained in $\{J=0\}$, $V_i\cap\{J=0\}$ has dimension zero, and
thus $V_i\setminus\{J=0\}$ is a nonempty Zariski open subset of $V_i$.
Furthermore, $\pi_S|_{V_i}$ is unramified precisely
at all points of this open subset.
According to \cite[Proposition T.8]{Iversen73}, the finite field extension $\F( 
S)\hookrightarrow \F(V_i)$ is separable, which proves (1).

As a consequence, for $1\le i\le t$ there exists a nonempty Zariski
open subset $U_i$ of $\cF^1$ such that
$\#(\pi_S|_{V_i})^{-1}(s)=\deg(\pi_S|_{V_i})$ for any $s\in U_i$. On
the other hand, $\#\pi^{-1}(s)=\deg\pi$ for any $s$ in a nonempty
Zariski open subset $U_S\subset\cF^1$. Let $s^*\in U_1\cap\cdots\cap
U_t\cap U_S\setminus\{s_{t+1},\ldots,s_u\}$ be such
 that $\pi_S|_{V_i}^{-1}(s^*)\cap \pi_S|_{V_j}^{-1}(s^*) =
\emptyset$ for $i\not=j$. We have
$$\deg\pi_S=\sum_{i=1}^t\deg \pi_S|_{V_i}=\sum_{i=1}^t
\#(\pi_S|_{V_i})^{-1}(s^*)=\#\pi_S^{-1}(s^*)=\#\pi^{-1}(s)=\deg\pi.$$

Finally, since $\{J=0\}$ intersects properly each irreducible
component $V_i$ with $1\le i\le t$, by \cite[Lemma 2.1]{GiMa19} we
conclude that $\mathcal{I}:J^\infty$ is a radical ideal.
\end{proof}

Corollary \ref{ViSeparable} and Lemma \ref{lemma:UpperBoundDegPi}
imply the bounds
\begin{equation}
\label{DdeltaBound}
D =\deg\pi_S = \deg \pi \leq \delta = \deg V \leq d^{d+e-1}.
\end{equation}
%
%
Now assume that $\bs \lambda \in \F^{d+e-1}$ is such that the
linear form $\mathcal{L}=\bs\lambda\cdot (\bs G,\bs H)\in \F[\bs
G,\bs H]$ induces a coordinate function $\ell$ of $\F[V]$ which is a
primitive element of the extension $\F(S)\hookrightarrow \F(V_i)$
for $1\le i\le t$. In particular, the minimal polynomial
$m_{\bs\lambda}\in \F[S,T]$ of $\ell$ satisfies $\deg_T
m_{\bs\lambda}=D$.

For any $s\in \cF^1$, we identify each point in the fiber
$\pi_S^{-1}(s)$ with the vector of its $(\bs g,\bs h)$-coordinates
and denote by $W_{s}\subset\cF^{d+e-1}$ the set of such
points, so that $\pi_S^{-1}(s)= W_{s} \times \{s\}$.
In particular,
$$\pi_S^{-1}(0)=W_0 \times \{0\} = \pi^{-1}(\bs \alpha,\bs 0) \times \{0\},
\quad \pi_S^{-1}(1)= W_1 \times \{1\} =
\pi^{-1}(\bs\alpha,\bs\beta) \times \{1\} .$$
Let $\bs\Lambda=(\Lambda_1,\ldots, \Lambda_{d+e-1})$ be a vector of
new indeterminates over $\F$.
\begin{lemma}\label{lemma:RandomChoiceLinearForm} There
exists a nonzero polynomial $R^{(4)} \in \F[\bs\Lambda]$ of degree at most
$4D^2$ such that for any $\bs\lambda\in \F^{d+e-1}$ with
$R^{(4)}  (\bs\lambda)\neq 0$, the linear form $\mathcal{L}=\bs\lambda\cdot
(\bs G,\bs H)$ separates the points of $W_0 \cup W_1$.
\end{lemma}
\begin{proof}
We consider the generic linear form
$\mathcal{L}_{\bs\Lambda}= \bs\Lambda\cdot (\bs G,\bs H)$ and define
$$R^{(4)}
=\displaystyle \prod_{p \neq q \in W_0\cup W_1}
(\mathcal{L}_{\bs\Lambda}(p)- \mathcal{L}_{\bs\Lambda}(q)) \in \F[\Lambda].$$
For $p \neq q$, we have $\mathcal{L}_{\bs\Lambda}(p) \neq \mathcal{L}_{\bs\Lambda}(q)$.
Since $\# W_0\cup W_1 \le 2D$ by Lemma \ref{lemma:FiberUnramified}, it follows that the nonzero polynomial $R^{(4)} \in
\F[\bs\Lambda]$ has degree at most $4D^2$. Any $\bs\lambda\in
\F^{d+e-1}$ not annihilating $R^{(4)} $ defines a linear form $\mathcal{L}_{\bs \lambda}$
that separates the points of $W_0 \cup W_1$.
\end{proof}
%
%
\section{An algorithm for decomposable interpolation}

The algorithm for computing an interpolant $(g,h)$ for the input
$(\bs \alpha, \bs \beta)$ with a generic $\bs \alpha$ proceeds as
follows. Computing in $\F[\bs G,\bs H,S]$ is straightforward, and we
can also take its elements modulo the ideal ${\mathcal I}(V)$ of the
homotopy curve $V$ to $\F[V] = \F[\bs G,\bs H,S]/ {\mathcal I}(V)$.
But testing whether an element of $\F[\bs G,\bs H,S]$ is zero in
$\F[V]$ amounts to an ideal membership test, which is a nontrivial
task. In order to circumvent this, we work with power series
expansions.
Recall the general assumption that $\char(\F)$ does not divide $d$.

From Section \ref{section:specialBeta}, we find an interpolating
polynomial $(g_0, h_0) \in W_0$ for $(\bs \alpha, \bs 0)$, so that
$\bs v = (g_0, h_0,0) \in V$. Now $\bs v$ actually lies on a single
irreducible component of $V$, say $\bs v \in V_1$. Furthermore, $S
\in \F[\bs G,\bs H,S]$ is a local parameter on $V_1$ at $\bs v$, so
that there is an injective morphism of $\F$-algebras $\tau \colon
{\mathcal O}_{\bs v} \to \F[\![Y]\!]$ from the local ring ${\mathcal
O}_{\bs v}$ of $V_1$ at $\bs v$ to the ring of power series in a new
variable $Y$. Since $\tau(S) = Y$ generates this ring, $\tau$ is
actually an isomorphism.

To simplify notation, we abbreviate $\bs G|\bs H = (G_{d-1}, \ldots,
G_0, H_{e-1}, \ldots, H_1)$, a vector of $d+e-1$ indeterminates. For
each $i < d$, $\gamma_i = \tau (G_i)  \in \F[\![Y]\!]$ is a power
series, and similarly $\eta_i = \tau (H_i)  \in \F[\![Y]\!]$ for $i
< e$. We abbreviate again $\bs\gamma|\bs\eta = (\gamma_{d-1},
\ldots, \gamma_0, \eta_{e-1}, \ldots, \eta_1)$.

A Newton-Hensel iteration computes an approximation of
$\bs\gamma|\bs\eta$ to a sufficiently large precision $N$, that is,
$\bs\psi \in (\F[Y]/(Y^N))^{d+e-1}$ with $\bs\psi \equiv
\bs\gamma|\bs\eta \bmod Y^N$. For a randomly chosen $\bs \lambda \in
\F^{d+e-1}$, we take the linear combination
$\bs\lambda\cdot(\bs\gamma|\bs\eta)=\sum_{1 \leq i \le d}
\lambda_i\gamma_{d-i}+\sum_{d< i < d+e} \lambda_i\eta_{d+e-i}$, its
truncation $L = \sum_{1 \leq i < d+e} \lambda_i \psi_i \in
\F[Y]/(Y^N)$ and find the minimal polynomial $m \in \F[Y,T]$ of
$\bs\lambda\cdot\bs\gamma|\bs\eta$, which equals the polynomial of
minimal degree in $T$ with $\deg m\le d^{d+e-1}$ such that $m(Y,L)
\equiv 0 \bmod Y^N$. Furthermore, we calculate $m'((\bs G|\bs H)_i)$
for each $i<e+d$, where $m' = \partial m /
\partial T$ is the derivative of $m$. These data provide a geometric
solution of $V_1$, and it then remains to substitute $S=1$ in them
to find the desired interpolation polynomial $(g,h)$ for $(\bs
\alpha, \bs \beta)$.

We sketch the whole algorithm as follows.
\begin{algorithm}
\label{mainAlgo}
${}^{}$

Input: Generic $(\bs \alpha, \bs \beta) \in \F^{d+e-1} \times
\F^{d+e-1}$, real $\epsilon$ with $0 < \epsilon <1$, access to a
finite set ${\mathcal S} \subseteq \F$ with at least $32
\epsilon^{-1} (d^2+de)^2d^{5(d+e-1)}$ elements.

Output:
A composite interpolating polynomial $f$ for
$(\bs \alpha, \bs \beta)$ as in Problem \ref{problem:interpolation}, or ``fail''.

\begin{enumerate}
\item
\label{step:initialPoint}%
Compute a point $(\bs g_0,\bs h_0)\in W_0$.
\item
\label{step:NHoperator}%
Set $M = d^{d+e-1}+1$ and $N = 2d^{2(d+e-1)} +1$. Compute truncated
power series approximations $\bs \gamma|\bs \eta \in
(\F[Y]/(Y^N))^{d+e-1}$ for
$\bs G|\bs H$, with $\lceil\log_2 N \rceil$ steps of the following
quadratic Newton-Hensel iteration for solving the equations $\bs
P=(P_1,\ldots,P_{d+e-1})=0$, where $P_j(\bs G,\bs H)= (G\circ
H)(\alpha_j)-\beta_j\, S \in \F[G, H, S]$ for $1\le j\le d+e-1$. The
initial value is $(\bs g_0,\bs h_0)$. One iteration step is given by
the Newton-Hensel operator $$N_{\bs P}(\bs \gamma,\bs \eta)=(\bs
\gamma,\bs \eta)-(J_{\bs P}^{-1}(\bs \gamma,\bs \eta,Y)) \bigg( \bs
P(\bs \gamma,\bs \eta,Y)^{\sf T} \bigg) \in \F[\![Y]\!]^{d+e-1},$$
where $J_{\bs P} = \partial \bs P / {\partial (G|H)} \in
\F[G,H,S]^{(d+e-1) \times (d+e-1)}$ is the Jacobian matrix of the
system and $\sf T$ denotes transposition. Let
$\psi 
\in (\F[Y]/(Y^N))^{d+e-1}$ be the resulting vector of power series
truncated up to order $N$.
\item
\label{step:chooseLambda}%
 Choose  $\bs\lambda\in \mathcal{S}^{d+e-1}$ uniformly at random.
For $1 \leq j < d+e$,
let
$$
\bs Z^{(j)} = (\lambda_1, \ldots, \lambda_{j-1},
\lambda_j+\Lambda_j, \lambda_{j+1}, \ldots, \lambda_{d+e-1}),
$$
where $\Lambda_1, \ldots, \Lambda_{d+e-1}$ are new indeterminates.
Compute
\begin{eqnarray*}
L^{(j)}&  = & \bs Z^{(j)} \cdot \bs \psi = \sum_{1 \leq i < d+e} \bs
Z^{(j)}_i \psi_i
\in \F[\Lambda_j][Y]/(Y^N),\\
&& (L^{(j)})^2, \ldots, (L^{(j)})^{2M} \in
\F[\Lambda_j][Y]/(Y^N,\Lambda_j^2).
\end{eqnarray*}
\item
\label{step:findMinPoly} For each $j\in \{1,\ldots,d+e-1\}$, find by
binary search the smallest value $k_j$ of $\ell$
for which
\begin{equation}
\label{binarySearch}
 \sum_{0 \leq i \le \ell} a_i^{(j)} (L^{(j)})^i  \equiv 0 \bmod
 (Y^N,\Lambda_j^2)
\end{equation}
has a solution $a_0^{(j)}, \ldots, a_{\ell}^{(j)} \in
\F[\Lambda_j][Y]$ with $\deg_Y a_i^{(j)} < M$ for $i \leq \ell$
and $a_\ell^{(j)}$ monic. This corresponds to a system of linear
equations for the coefficients in $\F[\Lambda_j]/(\Lambda_j^2)$ of
the $a_i^{(j)}$'s. Compute such a solution $a^{(j)}_0, \ldots,
a^{(j)}_{k_j}$ with $a^{(j)}_{k_j}$ monic of minimal degree $M_j$.
[Then $\ell \leq 2M$ and $M_j\le 2M$ for all $\ell$ and $M_j$
considered.]
\item
\label{step:geomSolution} For each $j\in \{1,\ldots,d+e-1\}$, let
\begin{align*}
m^{(j)} & =  \sum_{0 \leq i \leq k_j} a^{(j)}_i T^i \in
\F[\Lambda_j][Y,T],\quad  w^{(j)} = -\partial m^{(j)}/\partial
\Lambda_j \in \F[\Lambda_j][Y,T].
\end{align*}
Set
\begin{align*}
m & =  m^{(1)}|_{\Lambda_1 \leftarrow 0} \in \F[Y,T],\quad v_j =
w^{(j)}|_{\Lambda_j \leftarrow 0}+Y \cdot (\partial m/
\partial T)\in \F[Y,T].
\end{align*}
[Then $\bs \lambda, m, v_1, \ldots, v_{d+e-1}$ form a geometric
solution of $V_1$.]
\item
\label{step:subsS=1} Substitute 1 for $Y$ in the polynomials $m,
v_1, \ldots, v_{d+e-1}$. Let $m_1=m(1,T)$. Compute polynomials $u,
w_1, \ldots, w_{d+e-1} \in \F[T]$ of degrees less than $\deg m_1$
that satisfy
$$u \equiv (m_1')^{-1} \bmod m_1,\quad w_j \equiv (m_1')^{-1}\cdot v_j(1,T) \bmod m_1\ (1\le j<d+e).$$
[Then $\bs\lambda,m_1,w_1,\ldots, w_{d+e-1}\in \F[T]$ form a
geometric solution of the zero-dimensional variety
$(W_1\times\{1\})\cap V_1$.]
\item
\label{step:findgh}
Find a root $\eta$ of $m_1$ in some extension field $\K$ of $\F$ and compute
$$
g|h = (w_1(\eta),\ldots, w_{d+e-1}(\eta) ) \in \K^{d+e-1}.
$$
If $(g \circ h)(\bs \alpha_j) = \bs \beta_j$ for $1\leq j < d+e$,
then return $(g,h)$ else return ``fail''.
\end{enumerate}
\end{algorithm}

Details on performing some of the steps are given below.
We start with an illustration of the first steps of the algorithm.
\begin{example}
\label{example2}
Continuing Example \ref{example1}, we have
$d=e=2$, $\bs \alpha =
(5,6,7)$ and $\bs \beta=(3,3,3)$. Now we need this for $\bs \beta=(0,0,0)$
and so have to subtract $3$ from $g_0$. Then the very first solution yields
$(g_0,h_0) = (58, 840, -11).$
Then for $\bs\beta=(1,2,3)$, we have
\begin{align*}
J_{\bs P}(\bs g_0,\bs h_0)&= \left(\begin{array}{cccc}
h_0(\alpha_1) & 1 & g_0'(h_0(\alpha_1))\,\alpha_1\\
h_0(\alpha_2) & 1 & g_0'(h_0(\alpha_2))\,\alpha_2\\
h_0(\alpha_3) & 1 & g_0'(h_0(\alpha_3))\,\alpha_3
\end{array}\right)=
\left(\begin{array}{cccc}
-30 & 1 & -10\\
-30 & 1 & -12\\
-28 & 1 & \ 14
\end{array}\right),\\[1ex]
J_{\bs P}^{-1}(\bs g_0,\bs h_0)&=
\left(\begin{array}{cccc}
-\frac{13}{2} & 6 & \frac{1}{2}\\[0.5ex]
-189 & 175 & 15\\[0.5ex]
\frac{1}{2} & -\frac{1}{2} & 0
\end{array}\right),\\[1ex]
\bs P(\bs g_0,\bs h_0)&=\left(\begin{array}{cccc}
g_0\circ h_0(\alpha_1)-\beta_1 Y\\
g_0\circ h_0(\alpha_2)-\beta_2 Y\\
g_0\circ h_0(\alpha_3)-\beta_3 Y
\end{array}\right)=
\left(\begin{array}{cccc}
-\beta_1 Y\\
-\beta_2 Y\\
-\beta_3 Y
\end{array}\right)=
\left(\begin{array}{cccc}
-1\, Y\\
-2\, Y\\
-3\, Y
\end{array}\right), \\[1ex]
N_{\bs P}(\bs g_0,\bs h_0)&=
\left(\begin{array}{cccc}
58\\
840\\
-11
\end{array}\right)-
\left(\begin{array}{cccc}
-\frac{13}{2} & 6 & \frac{1}{2}\\[0.5ex]
-189 & 175 & 15\\
\frac{1}{2} & -\frac{1}{2} & 0
\end{array}\right)
\left(\begin{array}{cccc}
-1\, Y\\
-2\, Y\\
-3\, Y
\end{array}\right)\\
&= \left(\begin{array}{cccc}
58+7\, Y\\
840+206\, Y\\
-11-\frac{1}{2}\, Y
\end{array}\right)
\equiv
\left(\begin{array}{cccc}
\gamma_1\\
\gamma_0\\
\eta_1
\end{array}\right)
\bmod Y^2.
\end{align*}

Now suppose that the iteration stops here and in the algorithm, we choose $\bs \lambda = (-1,1,2)$.
Then we have
\begin{equation}
 L = -1 \cdot (58+7\, Y)  +
1 \cdot (840+206\, Y)  +2 \cdot (-11- \frac 1 2 Y) = 760 + 198 \, Y \in \F[Y].
\end{equation}
 \hfill \qed
\end{example}

We next provide some subroutines used in Algorithm \ref{mainAlgo} and estimate error probabilities and cost.
For the latter, we only count arithmetic operations in $\F$ and ignore the (small) cost
of Boolean operations and finding random elements.
Furthermore, we use the dense representation of (multivariate)
polynomials and mainly standard procedures for simplicity, except
that we denote by ${\sf M}(n)$ the cost of multiplying two
univariate polynomials of degree at most $n$, and by $\omega$ a
feasible exponent for matrix multiplication, so that two $n \times
n$ matrices can be multiplied with $n^\omega$ operations. Thus we
may use $\omega=2.3728639$ (Le Gall \cite{legall14}), and with the notation
$\softo(f) = \{g \colon g \in f (\log f)^{O(1)}\}$ for real
functions $f$ and $g$, we have ${\sf M} \in \softo(n)$. Possible
algorithmic improvements are discussed in Section \ref{section:cost
algorithm}.

For the analysis of the algorithm, we assume throughout that $\bs \alpha \in \F^{d+e-1}$ is generic, as required.

\begin{lemma}
\label{analysisFirstFourSteps} Steps \eqref{step:initialPoint} and
\eqref{step:NHoperator} can be performed with $\softo(d^{2(d+e-1)})$
operations in $\F$.
\end{lemma}
\begin{proof}
In step \eqref{step:initialPoint} we use Proposition
\ref{interpolForSmallSizeBeta} to compute a solution $(\bs g_0, \bs
h_0)$ for the interpolation problem given by $(\bs \alpha, \bs \beta
\cdot 0) = (\bs \alpha, \bs 0)$. This requires a $d$-refinement of
${\mathcal P}_{\bs 0}$, say ${\mathcal P} =\{  \{1\}, \ldots,
\{d-1\}, \{d,\ldots, d+e-1\} \}$. The system \eqref{matrixEquation2}
of linear equations can be solved for $\bs h_0$ with $(e-1)^\omega$
operations, and the interpolation for $\bs g_0$ with $O({\sf M}(d)
\log d)$ operations.

One step of the Newton-Hensel iteration in step
\eqref{step:NHoperator} at most doubles the degree of the $d+e-1$
power series in $Y$. In the first step, given in Example
\ref{example2}, this degree is at most $1$, and thus in the $k$th
step at most $2^k$. Writing $q=d+e-1$ for simplicity, the rows of
the $q \times q$ Jacobian matrix $J_{\bs P}$ are like the top rows
in \eqref{Jacobian dQ/dG,H}. At the start of the $k$th step, each
coefficient $(\bs g|\bs h)_i$ is replaced by a polynomial in $Y$ of
degree at most $2^{k-1}$, say by $(\bs \gamma^{(k-1)}|\bs
\eta^{(k-1)})_i$. The linear combination of the $\eta_i^{(k-1)}$
corresponding to $h(\alpha_j)$ with coefficients $\alpha_j^i$ can be
computed with $O(e2^k)$ operations, and then the powers
$h(\alpha_j)^2, \ldots, h(\alpha_j)^{d-1}$ with $O(de {\sf M}(2^k))$
further steps. All this needs to be done for the $q$ values of $j$,
at a cost of $O(deq {\sf M}(2^k))$ operations. Similarly, a single
$g'(h(\alpha_j))$, given the powers of $h(\alpha_j)$, corresponds to
at most $d$ multiplications of polynomials modulo $Y^{2^k}$ at a
cost of $O(d {\sf M}(2^k))$ operations. The multiplications by
powers of $\alpha_j$ are even cheaper. This leads to $O(dq {\sf
M}(2^k))$ operations.

Thus the matrix $J_{\bs P}(\bs \gamma^{(k-1)},\bs \eta^{(k-1)},Y)\in
\F[Y]^{q \times q}$ can be calculated with $O(de {\sf M}(2^k))$
operations. It is invertible modulo $Y$, hence also modulo
$Y^{2^k}$. Its inverse in $\F[Y]^{q \times q}$ can be found with
$O(q^\omega {\sf M}(2^k))$ operations in $\F$ and the product by
$\bs P(\bs \gamma^{(k-1)},\bs\eta^{(k-1)},S)^{\sf T}$ with $q^2 {\sf
M}(2^k)$ operations.

In total, the $k$th step of the iteration takes $q^\omega {\sf
M}(2^k)$ operations. Since ${\sf M}(n)$ is essentially linear,
$\sum_{1\leq k \leq \kappa} {\sf M}(2^k) \in O({\sf M}(N))$ and the
total cost for step \eqref{step:NHoperator} is $O(q^\omega {\sf
M}(N))$.
\end{proof}

Now $\pi_S$ is unramified at $(\bs g_0,\bs h_0,0) \in V = V_1 \cup
\cdots \cup V_t$ if and only if $\pi$ is unramified at $(\bs g_0,\bs
h_0,\bs\alpha,0\cdot \bs \beta)=(\bs g_0,\bs h_0,\bs\alpha,\bs 0)$,
and the latter holds by Lemma \ref{lemma:unramifiedalpha0}. Hence
$(\bs g_0,\bs h_0,0)$ is contained in exactly one of the components.
Their numbering is arbitrary and we may assume that $(\bs g_0,\bs
h_0,0) \in V_1$.

In order to provide some details for step \eqref{step:findMinPoly},
we first describe an algorithm that works for indeterminates
$\bs\Lambda$ instead of values $\bs\lambda \in \F^{d+e-1}$. Its
intermediate results may become too large and it is not suited for
implementation, but specializations of it are. So we consider a
polynomial ring $R = \F[\bs\Lambda]$ in variables $\bs\Lambda =
(\Lambda_1, \ldots, \Lambda_{d+e-1})$, $\bs\psi \in
(\F[Y]/(Y^N))^{d+e-1}$, and $\bs\Lambda \cdot \bs\psi = \sum_{1 \leq
i < d+e} \Lambda_i \psi_i \in \F[\bs\Lambda][Y]/(Y^{N})$.
\begin{subroutine}
\label{GaussAlgo}
${}^{}$

Input: $\bs\psi \in (\F[Y]/(Y^N))^{d+e-1}$.

Output: Integer $k$ and a polynomial $\mu \in \F(\bs\Lambda)[Y,T]$
with $\deg_T \mu = k$.
\begin{enumerate}
\renewcommand{\theenumi}{S\arabic{enumi}}
\item
\label{step:binaryLoop} Find the smallest value $k$ of $\ell$
for which
\begin{equation}
\label{binarySearchLambda} \sum_{0 \leq i \le \ell} A_i (\bs\Lambda
\cdot \bs\psi)^i
 \equiv 0 \bmod Y^N
\end{equation}
has a solution $A_0, \ldots, A_{\ell}\in \F(\bs\Lambda)[Y]$ with
$\deg_Y A_i < M$ for all $i\leq\ell$ and $A_\ell$ monic,
and compute such a solution $A_0, \ldots, A_{k}$ with $A_{k}$ monic
of minimal degree $M'$. Starting with $\ell = 1, 2, 4, \ldots$,
this binary
search keeps on doubling the values of $\ell$
while calls to the subroutine \eqref{step:Gauss} return ``fail''.
Then it backtracks and homes in on the smallest values $k$ and
similarly finds the smallest $M'$ for which \eqref{step:Gauss}
returns an output.
\item
\label{step:Gauss}%
The equation \eqref{binarySearchLambda} corresponds to a system of
$N$ linear equations for the $\ell M+M'$ coefficients of all $A_i
\in \F(\bs\Lambda)[Y]$ with $\deg_Y A_i<M$ for $i<\ell$ and
$\deg_YA_\ell< M'$. With a suitable matrix $B \in \F(\bs\Lambda)^{N
\times (\ell M'+M)}$, the vector $A \in \F(\bs\Lambda)^{\ell M'+M}$,
and a suitable constant vector $C \in \F(\bs\Lambda)^N$,
this corresponds to the system $BA=C$. Compute with a version of
Gaussian elimination matrices $D \in \F(\bs\Lambda)^{N\times (\ell
M+M')}$, $P \in \F^{(\ell M+M') \times (\ell M+ M')}$ and a vector
$E \in \F(\bs\Lambda)^N$ so that $B X = C$ is equivalent to $DP X =
E$ for all $X \in \F^{\ell M+M'}$, $P$ is a permutation
matrix, and $D$ and $E$ have the form:
$$ \begin{array}{|cccc|c|c|c|}
\cline{1-5} \cline{7-7}
 \star & * & * & * & * & & * \\
 & \ddots & \vdots & \vdots & * & & *\\
  & {0} & \ddots & \vdots &  * & & * \\
 &  & &  \star &  * & & *\\
\cline{1-5} \cline{7-7}
 & 0 & & & 0 & & *\\
\cline{1-5} \cline{7-7}
\end{array} \, , $$
where $*$ stands for an arbitrary value and $\star$ for a nonzero
value in $\F(\bs\Lambda)$. If the bottom entries of $E$ (designated
by a single $*$) are all zero, then compute by back-substitution a
solution $Y\in \F(\bs\Lambda)^{\ell M+M'}$ of $DY = E$ with the
value $0$ for those entries of $Y$ that correspond to the lower
block of the matrix, $A = P^{-1} Y \in \F(\bs\Lambda)^{\ell M+M'}$
and return $A$, else return ``fail''.
\item
\label{step:minpolLambda} Return $k$ and  $\mu = \sum_{0 \leq i \leq
k} A_i T^i \in \F(\bs\Lambda)[Y,T]$.
\end{enumerate}
\end{subroutine}

The Gaussian elimination works row by row, choosing a suitable nonzero pivot element $p \in \F(\Lambda)$,
multiplying one row by the pivot and then adding it to another row.
This creates entries with value $0$ in the matrix.

Now step \eqref{step:findMinPoly} of Algorithm \ref{mainAlgo} is
implemented by running Subroutine \ref{GaussAlgo} on the special
inputs $\bs Z^{(j)}$. Gauss elimination performs two types of
computations: calculations in the ground field, here
$\F(\bs\Lambda)$, and tests for zero. Such algorithms are sometimes
called \emph{arithmetic-Boolean circuits}. If $\bs\Lambda$ is
substituted by some special value, say $\bs\lambda$ or $\bs
Z^{(j)}$, then some of the tests that originally return ``$\neq 0$''
might return ``$= 0$''. In \eqref{R6def} below we define a
polynomial $R^{(5)} \in \F[\bs\Lambda]$ such that  $R^{(5)}
(\bs\lambda) \neq 0$ guarantees that this does not happen. We call
some $\bs\lambda \in \F^{d+e-1}$ \emph{lucky} if $R^{(5)}
(\bs\lambda) \neq 0$, and show that the latter occurs with high
probability.

In other words, the condition $R^{(5)} (\bs\lambda) \neq 0$ fixes
the Boolean sequence of test outcomes and turns the
arithmetic-Boolean circuit into an arithmetic circuit (or
straight-line program).

\begin{lemma}
\label{pivotEstimate} For the output of Subroutine \ref{GaussAlgo},
we have $k \leq d^{d+e-1}$. $R^{(5)} \in \F[\bs\Lambda]$ is a
nonzero polynomial of degree at most $32 (d^2+de)^2 d^{5(d+e-1)}$. A
$\bs\lambda$ as chosen in step \eqref{step:chooseLambda} of
Algorithm \ref{mainAlgo} is lucky with probability at least
$1-\epsilon$. For any lucky $\bs\lambda$, the following statements
hold:
\begin{enumerate}
\item
\label{lambdaSeparates} The linear form $\mathcal{L}=\bs\lambda\cdot
(\bs G,\bs H) \in \F[\bs G,\bs H]$ separates the points of $W_0 \cup
W_1$.
\item
\label{specialization} For $1\leq j < d+e$, we have $k_j = k$,
$m^{(j)} = \mu|_{\bs\Lambda \leftarrow \bs Z^{(j)}}$ and $m =
\mu|_{\bs\Lambda \leftarrow \bs\lambda}$
\end{enumerate}
\end{lemma}

\begin{proof}
From \cite[Proposition 1]{Schost03} we know that there exists a
solution to \eqref{binarySearch} with $k \leq D \leq d^{d+e-1}$ and
$M'\le \delta\le d^{d+e-1}$.
Therefore $\max\{k,M'\} \leq d^{d+e-1}< M$ in
\eqref{binarySearchLambda}, so that $\max\{\ell,M'\}$ never exceeds
$2M$.

For the probabilistic analysis in (\ref{lambdaSeparates}),
Lemma \ref{lemma:RandomChoiceLinearForm}
provides a nonzero polynomial $R^{(4)} \in \F[\bs\Lambda]$ of total degree at most
$4D^2 \leq 4 d^{2(d+e-1)}$ so that
$R^{(4)} (\bs\lambda)\neq 0$ is sufficient for the linear form $\mathcal{L}=\bs\lambda\cdot
(\bs G,\bs H)$ to separate the points of $W_0 \cup W_1$.

We now present a similar polynomial $R^{(5)}   \in \F[\bs\Lambda]$
for the claim in (\ref{specialization}), namely, such that for any
$\bs \lambda \in \F^{d+e-1}$ with $R^{(5)} (\bs\lambda) \neq 0$, all
pivots and denominators in the subroutine with $\bs\lambda$
substituted for $\bs\Lambda$ are nonzero.
A fortiori, this holds for substitution by any $\bs Z^{(j)}$.

We represent each intermediate result in $\F(\bs \Lambda)$ as the
quotient of two coprime polynomials in $\F[\bs \Lambda]$. The
\emph{degree} of the rational function is the maximum of the two
polynomial degrees. Each calculation step yields first a quotient of
two arbitrary polynomials which we then simplify by dividing out
their $\gcd$. All pivots and divisors in the Gaussian elimination of
step \eqref{step:Gauss} for some value of $\ell$ are quotients of
determinants of two square submatrices
of the original coefficient matrix $B$;
see Edmonds \cite{edm67} and Bareiss \cite{bar68}.
In \eqref{binarySearchLambda}, the (total) degree in $\bs\Lambda$ is
at most $\ell$, so this also holds for all entries in $B$. The size
of square submatrices is at most $N \times N$, and their
determinants have degree at most $\ell N$ in $\bs\Lambda$. We denote
by $R^{(5)}_{\ell,M'} \in \F[\bs\Lambda]$ the product of all pivots
examined and found to be nonzero in the course of the subroutine.
Since there are at most $N$ such pivots, we have $\deg
R^{(5)}_{\ell,M'} \leq \ell N^2 \leq 2d^{d+e-1} N^2$. At most $2
\log_2 2d^{d+e-1} \leq 2(d+e)\log_2 d$ values of $\ell$ and $M'$ are
used. We let
 \begin{equation}
 \label{R6def}
 R^{(5)} =  R^{(4)} \cdot \prod_{\ell,M'} R^{(5)}_{\ell,M'},
 \end{equation}
where the product runs  over all such $\ell$. $ R^{(5)}$  is a
nonzero polynomial of degree
 at most
 \begin{align*}
4 d^{2(d+e-1)} +& 4 (\log_2 d)^2  (d+e)^2 \cdot 2d^{d+e-1} N^2 \\
&  \leq  32  (d^2+de)^2\cdot d^{5(d+e-1)}.
 \end{align*}

It is well known (see for example von zur Gathen \&\ Gerhard {\cite[Lemma 6.44]{GaGe99}}) that for a nonzero polynomial
$p \in \F[Y_1, \ldots, Y_n]$
and a uniformly random choice of $y = (y_1, \ldots, y_n) \in {\mathcal S}^n \subseteq \F^n$,
$$
{\text {prob}} \{p(y) = 0\} \leq \deg p / \#{\mathcal S},
$$
where $\deg$ denotes the total degree.
This implies the second claim.
\end{proof}

In particular, in step \eqref{step:geomSolution} of Algorithm
\ref{mainAlgo}, we have $m=  m^{(j)}|_{\Lambda_j \leftarrow 0}$ for
any $j < d+e$.

\begin{lemma}\label{lemma:ComplexityPowersL}
Step \eqref{step:chooseLambda}
can be performed with $\softo(d^{3(d+e-1)})$ operations in
$\F$ for $1\le j< d+e$.
\end{lemma}

\begin{proof}
We perform $(d+e) M$ multiplications of univariate polynomials in
$\F[\Lambda_j][Y]$, each with degree at most $1$ in $\Lambda_j$ and
at most $N$ in $Y$. A single such multiplication can be done with
$\softo(N)$ operations, for a total cost of $\softo(d^{3(d+e-1)})$
operations in $\F$.
\end{proof}

\begin{lemma}\label{lemma:ComplexityFindMinPoly}
For a lucky choice of $\bs\lambda$, step \eqref{step:findMinPoly}
can be performed with $\softo(d^{6(d+e-1)})$ operations in
$\F$, for $1\le j<d+e$.
\end{lemma}

\begin{proof}
If \eqref{binarySearch} has a solution for some values $\ell$ and
$M'$, then it also has one for $\ell+1$ and $M'+1$ and all larger
values of $\ell$ and $M'$. Therefore, by Lemma \ref{pivotEstimate},
in the binary search in step \eqref{step:findMinPoly} all values
$\ell$ and $M'$ are at most $2 d^{d+e-1}$.

We first consider \eqref{binarySearch} for an arbitrary value of
$\ell$ and $M'$ and write each unknown $a^{(j)}_i \in
\F[\Lambda_j][Y]$ as $a^{(j)}_i = \sum_{0 \leq m < M} a^{(j)}_{im}
Y^m$ for $i<\ell$ and $a^{(j)}_\ell = \sum_{0 \leq m < M'}
a^{(j)}_{im} Y^m$, and $(L^{(j)})^i = \sum_{0 \leq h < N}
\ell^{(j)}_{ih} Y^h \bmod (Y^N,\Lambda_j^2)$, with all
$\ell^{(j)}_{ih} \in \F[\Lambda_j]$ polynomials of degree at most 1
and unknowns $a^{(j)}_{im}$ taking values in
$\F[\Lambda_j]/(\Lambda_j^2)$. Thus \eqref{binarySearch} corresponds
to a system of $N$ linear equations in $\ell M+M'$ unknowns over
$\F[\Lambda_j]/(\Lambda_j^2)$.

We use the calculation of Subroutine \ref{GaussAlgo} with input $\bs
Z^{(j)}$ instead of $\bs\Lambda$.
By the assumption that the choice of $\bs\lambda$ is lucky, all rational
functions of $\F(\Lambda_j)$ arising during the execution of
Subroutine \ref{GaussAlgo} on input $\bs Z^{(j)}$ are well-defined
modulo $\Lambda_j$, which implies that they are well-defined modulo
$\Lambda_j^2$. We conclude that the output of step
\eqref{step:findMinPoly} is of the form
\begin{equation}\label{eq:expression m(j)moduloLambda2}
m^{(j)}|_{\Lambda_j \leftarrow 0}+ (\partial
m^{(j)}/\partial\Lambda_j)|_{\Lambda_j \leftarrow 0}\cdot
\Lambda_j.\end{equation}

The system of linear equations is of size $(\ell M +M')\times N$ and
$\ell M+M', N \leq 3 d^{2(d+e-1)}$. Solving it can be done with
$27d^{6(d+e-1)}$ calculations on matrix entries,
which are elements of the quotient ring $\F[\Lambda_j]/(\Lambda_j^2)$.
More precisely, each operation involves the images modulo
$\Lambda_j^2$ of two rational functions of $\F(\Lambda_j)$ (which
are well-defined modulo $\Lambda_j^2$), and is performed modulo
$\Lambda_j^2$.
Since each operation in $\F(\Lambda_j)$ is replaced by $O(1)$
operations in $\F$, we conclude that the whole procedure requires
$O(d^{6(d+e-1)})$ operations in $\F$. At most $2 \log_2 (2
d^{d+e-1})$ values of $\ell$ and $M'$ occur and the claimed bound on
the running time follows.
\end{proof}

\begin{lemma}
For a lucky choice of $\bs\lambda$, the values computed in step
\eqref{step:geomSolution} constitute a geometric solution of $V_1$.
The computation can be performed with $\softo(d^{2(d+e-1)})$
operations in $\F$.
\end{lemma}
\begin{proof}
For $1\le j<d+e$, the polynomial $m^{(j)}$ has $\ell M+M'\leq 3
d^{2(d+e-1)}$ coefficients in $\F[\Lambda_j]/(\Lambda_j^2)$. By the
expression \eqref{eq:expression m(j)moduloLambda2} of $m^{(j)}$ the
claim on the number of operations in $\F$ follows. The fact that the
output of step \eqref{step:geomSolution} is a geometric solution of
$V_1$ is due to \eqref{eq:definitionParametrizChow} and
\eqref{eq:parametrizacionesChow}.
\end{proof}

\begin{lemma}
For a lucky choice of $\bs\lambda$, the output of step
\eqref{step:subsS=1} is a geometric solution of the zero-dimensional
variety $(W_1\times\{1\})\cap V_1$. It can be performed with
$\softo(d^{3(d+e-1)})$ operations in $\F$.
\end{lemma}
\begin{proof}
Recall that $(\pi_S|_{V_1})^{-1}(1)=(W_1\times\{1\})\cap V_1$.
Observe that $\#((W_1\times\{1\})\cap V_1)=\deg \pi_S|_{V_1}=D_1>0$
and all the elements of $\pi_S((W_1\times\{1\})\cap V_1)$ are
interpolants for the interpolation problem defined by
$(\bs\alpha,\bs\beta)$.
Since $\bs\lambda,m,v_1,\ldots,v_{d+e-1}$ form a geometric solution
of $V_1$, for $m_1=m(1,T)$ we have
\begin{align*}
(W_1\times\{1\})\cap V_1 = \{&(1,\bs g,\bs h) \in \F^{d+e-1} \colon
\
m_1(\mathcal{L}(\bs g,\bs h))=0,\\
&\ m_1'(\mathcal{L}(\bs g,\bs h)) (g|h)_j=v_{j}(1,\mathcal{L}(\bs g,\bs
h))\ (1\le j <d+e)\}.
\end{align*}
Observe that $\deg m_1\le D_1$, $m_1\not=0$ and $m_1$ vanishes on
the set $\{\mathcal{L}(\bs g,\bs h) \colon (\bs g,\bs h)\in W_1\}$,
of cardinality $D_1$. It follows that $\deg m_1=D_1$, $m_1$ and
$m_1'$ are coprime, and thus the inverse $(m_1')^{-1}$ modulo $m_1$
is well-defined. In particular, defining $w_j=(m_1')^{-1}\cdot v_j$
modulo $m_1$ for $1\le j<d+e$, we obtain
$$
(W_1\times\{1\})\cap V_1=\{(1,w_1(\eta),\ldots,w_{d+e-1}(\eta))\in
\F^{d+e} \colon \ \eta\in\cF,\ m_1(\eta)=0\}.
$$
This shows that $\bs\lambda,m_1,w_1,\ldots,w_{d+e-1}$ form a
geometric solution of $W_1$.

The dense representation of $m(1,T),v_1(1,T),\ldots,v_{d+e-1}(1,T)$
can be obtained from that of $m(Y,T),v_1(Y,T),\ldots,v_{d+e-1}(Y,T)$
with $O((d+e)MN)\subseteq\softo(d^{3(d+e-1)})$ operations in $\F$.
The remaining computations consist of a modular inversion and
$d+e-1$ modular multiplications of univariate polynomials of degree
at most $M$, which contribute with $\softo(d^{d+e-1})$ additional
operations.
\end{proof}

Step (\ref{step:findgh}) of Algorithm \ref{mainAlgo} has to find a
root of $m_1\in\F[T]$, with $u = \deg m_1 \leq d^{d+e-1}$. A natural
implementation is to find an irreducible factor $m_2$ of $m_1$ and
use the root $T \bmod m_2 \in \K = \F[T]/(m_2)$. While the cost of
all other steps can be analyzed for general $\F$, this is not clear
here. Over a finite field $\F_q$, the factorization can be done
probabilistically with $\softo(u^{1.5} \log q + u \log^2 q)$ steps
(Kedlaya \&\ Umans \cite{keduma11}). For $\F=\Q$, efficient factorization algorithms
are known, but their cost estimates require a bound on the size of
the coefficients of $m_1$. Such a bound can be derived along the
computations of Algorithm \ref{mainAlgo}, but we have not done so.
Thus in the following we use some ``cost of root finding at degree
$u$'' for which we have no good estimate in general.

\begin{lemma}
\label{cost: findgh} Step (\ref{step:findgh}) of Algorithm
\ref{mainAlgo} can be performed with $\softo(d^{2(d+e-1)})$
operations in $\F$ plus the cost of root finding at degree not more
than $d^{d+e-1}$.
\end{lemma}

\begin{proof}
For any root $\eta$ of $m_1\in\F[T]$, $\K = \F[\eta]$ has degree at
most $d^{d+e-1}$ over $\F$. For any $j<d-e$, the same bound holds
for the degree of $w_j \in \F[T]$, so that all $w_j(\eta) \in\K$ can
be computed with $\softo(d^{d+e-1})$ operations in $\F$. Then
$\softo(d^{2(d+e-1)})$ operations in $\F$ are sufficient to test
whether $(g \circ h)(\bs \alpha_j) = \bs \beta_j$ for $1\leq j <
d+e$.
\end{proof}

\begin{theorem}
\label{thm:final} Let $d, e \geq 2$,  $0 < \epsilon <1$, $\F$ be a
field of characteristic not dividing $d$ and with at least
$32\epsilon^{-1}(d^2+de)^2{d^{5(d+e-1)}}$ elements, $\bs \alpha\in
\F^{d+e-1}$ generic, and $\bs \beta \in\F^{d+e-1}$. Then Algorithm
\ref{mainAlgo} returns a solution $(g,h)$ to Problem
\ref{problem:interpolation}, so that $(g \circ h)(\bs \alpha_j) =
\bs \beta_j$ for $1\leq j < d+e$, with probability at least
$1-\epsilon$. It uses $\softo(d^{6(d+e-1)})$ operations in $\F$ plus
root finding for degree at most $d^{d+e-1}$.
\end{theorem}

\begin{proof}
The various estimates in Lemmas \ref{analysisFirstFourSteps} through
\ref{cost: findgh} imply a cost of $\softo(d^{6(d+e-1)})$ operations
in $\F$ plus root finding. For a lucky choice of $\bs\lambda$, the
output is indeed a solution to Problem \ref{problem:interpolation},
and the probability for this to happen is at least $1-\epsilon$ by
Lemma \ref{pivotEstimate}.
\end{proof}
%
%
\section{Cost of the algorithm}\label{section:cost algorithm}

We have not tried to optimize the cost of our method.
The input to Problem \ref{problem:interpolation} consists of $2(d+e-1)$ elements of $\F$.
We do not expect that runtime polynomial in this input size is feasible,
since the coefficients of a solution $(g,h)$ lie in a field of potentially exponential degree.

The more appropriate view seems to consider
the geometric variant of Problem \ref{problem:interpolation} as
stated in Problem \ref{problem:interpolation2}. This involves the
incidence variety $\Gamma_{n,d}$ and its projection $\pi$ to
$\cF^{2(d+e-1)}$. According to \eqref{DdeltaBound}, their geometric
degrees are bounded as $D = \deg \pi \leq \delta = \deg \Gamma_{n,d}
\leq d^{d+e-1}$. In step (\ref{step:NHoperator}) of Algorithm
\ref{mainAlgo}, the quantities $M$ and $N$ are defined so that $M
>\delta\ge D$ and $N>2\delta D$. These inequalities are all that is being
used about $M$ and $N$ in the lemmas above, and the algorithm would
also work with $D+1$ and $2\delta D +1$ in their place.

Within the algorithm, we avoid using the values of $\delta$ and $D$ because we do not know them.
But we could start the algorithm with some small guess at those values
and then keep on doubling until we actually find a solution $(g,h)$.
The running time then is polynomial in $\delta$.
%

More precisely, the runtime of this variant of our algorithm is
$\softo{((d+e)^{\omega}D\delta}+(D\delta)^\omega+(d+e)D^2\delta)$
plus root finding for degree at most $D$. In particular, Lemma
\ref{lemma:ComplexityFindMinPoly} can be modified so that for a
lucky choice of $\bs\lambda$, step \eqref{step:findMinPoly} can be
performed with
$\softo((d+e)D^{\omega}\delta)\subset\softo(d^{4(d+e-1)})$
operations in $\F$. We briefly discuss below this modification.

The important papers of Lecerf \cite{lec03} and van der Hoeven \&\ Lecerf \cite{hoelec19} present
algorithms for solving systems of polynomial equations, with running
times $\softo{(B^3)}$ and $\softo{(B^2)}$, respectively, where $B$
is the system's B\'ezout number. The values $D$ and $\delta$
considered above satisfy $D\le\delta \leq B$. Our approach can
approximate $D$ and $\delta$ by binary search, as described. In
cases where they are considerably smaller than $B$, the running time
presented here may be smaller than that of the cited works.

The paper of \cite{hoelec19} assumes that their elimination encounters only radical ideals
and that the polynomial system forms a regular sequence.
In our case, the latter requires a saturation with respect to the Jacobian.
This has degree up to roughly $n^{d+e}$; see Figure \ref{Jacobian dQ/dG,H}.
If this is added to the equations of Problem \ref{problem:interpolation2}
via the Rabinovich trick, then it will increase the B\'ezout number by the stated factor.
Our approach circumvents this problem by using a starting value for the Newton iteration
where the Jacobian does not vanish.

For the characteristic $p$, we only demand that it does not divide
the degree $n$ of the target polynomial. \cite{lec03} is stated only
for characteristic zero.

We discuss here \cite{lec03} and \cite{hoelec19}, because they yield important progress
in general. But it does not seem that they are applicable to our problem.
%
%
\subsection{Hermite--Pad\'e approximations to perform step \eqref{step:findMinPoly} of Algorithm \ref{mainAlgo}}
\label{subsection:ComputationMinPolynomial}
In this section we briefly discuss a variant to perform step
\eqref{step:findMinPoly} of Algorithm \ref{mainAlgo}. Recall that,
in such a step, for each $j\in \{1,\ldots,d+e-1\}$ we find by binary
search the smallest values $k_j$ of $\ell$ and, given $k_j$, of
$M_j$ for $M'$ for which
\begin{equation}
\label{binarySearchbis} \sum_{0 \leq i \le \ell} a_i^{(j)}
(L^{(j)})^i \equiv 0 \bmod
 (Y^N,\Lambda_j^2)
\end{equation}
has a solution $a_0^{(j)}, \ldots, a_{\ell}^{(j)} \in
\F[\Lambda_j][Y]$ with $\deg_Y a_i^{(j)}\le \delta$ for $i<\ell$,
$\deg_Y a_\ell^{(j)} < M'$ and $a_\ell^{(j)}$ monic. For this
purpose, we consider \eqref{binarySearchbis} as a system of linear
equations for the coefficients in $\F[\Lambda_j]/(\Lambda_j^2)$ of
the $a_i^{(j)}$'s, and compute such a solution $a^{(j)}_0, \ldots,
a^{(j)}_{k_j}$ with $a^{(j)}_{k_j}=1$ monic of degree $M_j'$.

To find such a solution of \eqref{binarySearchbis}, we interpret it
as a problem of Hermite-Pad\'e
approximation, which can be solved applying an algorithm of Bostan et al.\ \cite{BoJeSc08}, based
on fast linear-algebra algorithms for matrices of fixed displacement
rank (cf.\ Bini \&\ Pan \cite{BiPa94},
Pan \cite{Pan01}).
More precisely, for a suitable ordering of the unknowns, the matrix
$B_j$ of system \eqref{binarySearchbis} is a block-Toeplitz matrix
with $\ell$ blocks (see Cafure et al.\ \cite[Lemma 4.3]{CaMaWa06}; see also
\cite[Lemma 17]{BoJeSc08}).

From \cite[Corollary 1]{BoJeSc08} it follows that, if there exist a
nonzero solution of \eqref{binarySearchbis}, then one such solution
can be computed with $\softo\big(\ell^{\omega-1}{\sf M}(\ell
M')\big)$ arithmetic operations in $\F$ and error probability at
most $1/(2\epsilon)$ for any fixed $\epsilon>0$. The algorithms of
\cite{BoJeSc08} work over a field. We apply their results to
computations over $\F(\!(\Lambda_j)\!)$, but perform calculations in
$\F[\![\Lambda_j]\!]/(\Lambda_j^2)$, which is all that we need. This
works because our conditions guarantee that all divisions are by
series which are invertible in $\F[\![\Lambda_j]\!]/(\Lambda_j^2)$.
Thus computing in $\F[\Lambda_j]/(\Lambda_j^2)$ is sufficient.
Taking into account that at most $2 \log_2 (2 D)$ values of $\ell$
and $2 \log_2 (2 \delta)$ values of $M'$ occur for $1\le j\le
d+e-1$, we have the following result.
\begin{lemma}\label{lemma:complexityHermitePade}
For a lucky choice of $\bs\lambda$, step \eqref{step:findMinPoly} of
Algorithm \ref{mainAlgo} can be computed with
$\softo\big((d+e)D^{\omega}\delta\big)$ operations in $\F$ and error
probability at most $1/(2\epsilon)$ for any fixed $\epsilon>0$.
\end{lemma}

Next we analyze the probability of failure of the choice of the
vector $\bs\lambda$.
\begin{proposition}\label{prop:computationgeosolcurve}
Let $\mathcal{S}\subset\F$ be a finite set with at least
$64\epsilon^{-1}(d+e) D^2(\delta+2)^{\log\delta+3}$ elements. A
$\bs\lambda$ chosen uniformly at random in $\mathcal{S}^{d+e-1}$ is
lucky with probability at least $1-2\epsilon$.
\end{proposition}
\begin{proof}
Arguing as in Lemma \ref{pivotEstimate}, we analyze the denominators
in $\F[\Lambda_j]$ which occur when using the algorithm of
\cite{BoJeSc08} to solve the Hermite-Pad\'e approximation problem
\eqref{binarySearchbis} for $1\le j\le d+e-1$. This algorithm is an
adaptation of Kaltofen's \emph{Leading Principal Inverse} algorithm
(Kaltofen \cite{Kaltofen94}, \cite{Kaltofen95}). The algorithm
performs a recursive reduction of the computation of the inverse of
a ``generic-rank-profile'' square input matrix
$$A=\left(%
\begin{array}{cc}
  A_{1,1} & A_{1,2} \\
  A_{2,1} & A_{2,2} \\
\end{array}%
\right)$$
to that of the leading principal submatrix $A_{1,1}$ and its Schur
complement $\Delta=A_{2,2}-A_{2,1}A_{1,1}^{-1}A_{1,2}$. The
divisions which arise during the execution of this recursive step
are related to the computation of $A_{1,1}^{-1}$ and $\Delta^{-1}$
and a routine of ``compression'' (cf.\ \cite[Problem
2.2.11.c]{BiPa94}) of the generators of matrices which are obtained
as certain products involving $A_{1,1}^{-1}$, $\Delta^{-1}$,
$A_{1,2}$ and $A_{2,1}$. The latter in turn requires the computation
of the inverses of certain submatrices of the products under
consideration.

Each entry of the matrix $B_j$ of the linear system
\eqref{binarySearchbis} is a coefficient of a power $(L^{(j)})^i$,
which is therefore a polynomial in $\Lambda_j$ of degree at most
$i\le \ell$. Since the generic-rank-profile matrix $A$ is obtained
by multiplying the matrix $B_j$ with suitable matrices with entries
in $\F$, we conclude that the entries of $A_{1,1}$ are polynomials
of $\F[\Lambda_j]$ of degree at most $\ell\le D$, while the
numerators and denominators of the entries of $\Delta^{-1}$ are
polynomials of $\F[\Lambda_j]$ of degree at most $D(\delta+2)$.
Therefore, by a simple recursive argument it can be seen that the
numerators and denominators of all leading principal submatrices and
Schur complements which are inverted during the algorithm have
degrees bounded by $\mathcal{D}=D(\delta+2)^{\log\delta}$. This in
turn implies that the denominators arising during the compression
routine have degrees bounded by $3\mathcal{D}\delta$.

Taking into account that the algorithm of \cite{BoJeSc08} for each
value of $\ell$ and $M'$ consists of at most $\log_2 \delta$
recursive steps, and that each recursive step requires the inversion
of at most 4 matrices, we conclude that the product of all the
denominators arising during the algorithm has degree bounded by
$8\mathcal{D}\delta\log_2\delta$. Since at most $2 \log_2 (2 D)$
values of $\ell$ and $2 \log_2 (2 \delta)$ values of $M'$ occur for
$1\le j\le d+e-1$, we conclude that the degree of all denominators
is bounded from above by
$32(d+e)\mathcal{D}\delta\log_2^3(2\delta)$.

Finally, according to Lemma \ref{lemma:RandomChoiceLinearForm},
$\bs\lambda$ must not annihilate a polynomial $R^{(4)}$ of degree
$2D^2$. The statement of the proposition readily follows.
\end{proof}

We recall from \eqref{degVupperBound} and \eqref{Dbound} that
$D \leq \delta \leq d^{d+e-1}$.
Then from the above estimate, Lemma
\ref{lemma:complexityHermitePade} and the error probability
$1/(2\mu)$ of the algorithm underlying Lemma
\ref{lemma:complexityHermitePade},  we deduce the following result.
\begin{theorem}
\label{thm:final_D_delta} Let $d, e \geq 2$,  $0 < \epsilon <1$,
$\F$ be a field of characteristic not dividing $d$ and with at least
$64\epsilon^{-1}(d+e) D^2(\delta+2)^{\log\delta+3}$ elements, $\bs
\alpha\in \F^{d+e-1}$ generic, and $\bs \beta \in\F^{d+e-1}$. Then a
solution $(g,h)$ to Problem \ref{problem:interpolation} can be
computed with probability at least $1-\epsilon$. It uses
$\softo{((d+e)^{\omega}D^\omega\delta)}$ operations in $\F$ plus
root finding for degree at most $D$.
\end{theorem}

\section*{Open question}
Our algorithm uses time polynomial in the \emph{geometric input
size}, namely the degree of the variety $\Gamma_{n,d}$ in
\eqref{defGamma} that encodes our problem. The entries of an output
$(g,h)$ lie in an extension field whose degree is bounded by $D \leq
d^{d+e-1}$ and our runtime estimate is polynomial in this bound.
It remains open how close this bound is to the actual minimal extension degree required;
some experiments might shed light on this, but we have not done this.

One may wonder whether time polynomial in the actual output size of an instance is possible,
or whether some particularly succinct representation of the output allows a faster solution.
The special cases treated in Section \ref{section:specialBeta} even allow a solution
in time polynomial in the size of the input, consisting just of the $d+e-1$ field elements that
form $\alpha$ and $\beta$.
This cannot be expected in general.

Another approach would be to use a numerical instead of our algebraic homotopy.
We have not investigated this.

\section*{Acknowledgements}
We thank Adrien Touboul for pointing out to us the question addressed in this paper,
and Joris van der Hoeven and Gr\'egoire Lecerf for useful comments.

\providecommand{\bysame}{\leavevmode\hbox
to3em{\hrulefill}\thinspace}
\providecommand{\MR}{\relax\ifhmode\unskip\space\fi MR }
\providecommand{\MRhref}[2]{%
  \href{http://www.ams.org/mathscinet-getitem?mr=#1}{#2}
} \providecommand{\href}[2]{#2}


\begin{thebibliography}{10}

\bibitem{Abhyankar90}
S.~Abhyankar, \emph{Algebraic geometry for scientists and
engineers}, Math.
  Surv. Monogr., vol.~35, American Mathematical Society, Providence, RI, 1990.


\bibitem{bar68}
E.~Bareiss. \emph{Sylvester's Identity and Multistep
                  Integer-Preserving Gaussian Elimination},
Math. Comp. \textbf{22} (1968), 565--578.

\bibitem{bertas10}
D.~Berend and T.~Tassa, \emph{Improved bounds on Bell numbers and on
moments of sums of random variables}, Probability and Mathematical
Statistics, \textbf{30} (2010), 185--205.

\bibitem{BiPa94}
D.~Bini and V.~Pan, \emph{Polynomial and matrix computations},
Progress in
  Theoretical Computer Science, Birkh{\"a}user, Boston, 1994.

\bibitem{blagat13}
R.~Blankertz, J.~von~zur~Gathen, and
                  K.~Ziegler,
\emph{Compositions and collisions at degree $p^2$}, J. Symbolic
Computation \textbf{59} (2013), 113-145.

  \bibitem{BoJeSc08}
A.~Bostan, C.-P. Jeannerod, and E.~Schost, \emph{Solving structured
linear systems with large displacement rank}. Theoret.
  Comput. Sci. \textbf{407} (2008), 155--181.

\bibitem{CaMa06}
A.~Cafure and G.~Matera, \emph{Improved explicit
  estimates on the number of solutions of equations over a finite field},
  Finite Fields Appl. \textbf{12} (2006), 155--185.

\bibitem{CaMaWa06}
A.~Cafure, G.~Matera, and A.~Waissbein, \emph{Inverting bijective
  polynomial maps over finite fields}, in Proceedings of the 2006 Information
  Theory Workshop, ITW2006,
  G.~Seroussi and A.~Viola, eds., IEEE Information Theory Society, 2006,
  pp.~27--31.

\bibitem{Danilov94}
V.~Danilov, \emph{Algebraic varieties and schemes}, in
I.~Shafarevich (Ed.),
  Algebraic Geometry I, vol.~\textbf{23} of Encyclopaedia of Mathematical Sciences,
  Springer, Berlin Heidelberg New York, 1994, 167--307.

\bibitem{edm67}
J.~Edmonds, \emph{Systems of distinct representatives and linear
                  algebra},
J.~ Res.~Nat.~Bureau Standards \textbf{71B} (1967),
241--245.

\bibitem{Fulton84}
W.~Fulton, \emph{Intersection theory}, Springer, Berlin Heidelberg
New York,
  1984.

\bibitem{legall14}
F.~Le Gall, \emph{Powers of tensors and fast matrix multiplication},
in Proceedings of the 39th International Symposium on Symbolic and
Algebraic Computation (ISSAC 2014), 2014, pages 296-303. Also {\tt
arXiv:1401.7714}.

\bibitem{gat14}
J.~von~zur Gathen, \emph{{Normal form for Ritt's Second Theorem}},
{Finite Fields Appl.}
  \textbf{27} (2014), 41--71.

\bibitem{GaGe99}
J.~von~zur {Gathen} and J.~Gerhard, \emph{Modern computer algebra},
  Cambridge Univ. Press, Cambridge, 2013.

\bibitem{gatmat17}
J.~von~zur Gathen and G. Matera, \emph{Density of real and complex
decomposable univariate polynomials}, {Q. J. Math.} \textbf{68}
(2017), no.~4, 1227--1246.

\bibitem{GiMa19}
N.~Gim\'enez and G. Matera, \emph{On the bit complexity of
polynomial system solving}, J. Complexity \textbf{51} (2017),
20--67.

\bibitem{GiLeSa01}
M.~Giusti, G.~Lecerf, and B.~Salvy, \emph{A {Gr\"obner} free
alternative
  for polynomial system solving}, J. Complexity \textbf{17} (2001), 154--211.

\bibitem{Heintz83}
J.~Heintz, \emph{{Definability} and fast quantifier elimination in
  algebraically closed fields}, Theoret. Comput. Sci. \textbf{24} (1983), 239--277.

\bibitem{hoelec19}
J.\ van der Hoeven and G.\ Lecerf, \emph{On the complexity exponent
of polynomial system solving}, Found. Comput. Math. \textbf{21}
(2021), no.~1, 1--57.

\bibitem{lec03}
G.\ Lecerf,
\emph{Computing the equidimensional decomposition of an algebraic closed set by means of lifting fibers},
 J.\ Complexity, \textbf{19}(4), 2003, 564--596.

\bibitem{Iversen73}
B.~Iversen, \emph{Generic local structure of the morphisms in
commutative algebra}, Springer Lect. Notes in Math. 13, Springer,
New York, 1973.

\bibitem{Kaltofen94}
E.~Kaltofen, \emph{Asymptotically fast solution of {Toeplitz}--like
  singular linear systems}, in Proceedings ISSAC'94,
  J.~von~zur Gathen and M.~Giesbrecht, eds., New York, 1994, ACM Press,
  pp.~297--304.

\bibitem{Kaltofen95}
\leavevmode\vrule height 2pt depth -1.6pt width 23pt,
E.~Kaltofen, \emph{Analysis
of
  {Coppersmith}'s block {Wiedemann} algorithm for the parallel solution of
  sparse linear systems}, Math. Comp. \textbf{64} (1995), 777--806.

\bibitem{keduma11}
K.~Kedlaya and C.~Umans, \emph{Fast polynomial factorization and
modular composition}, SIAM J. Computing \textbf{40} (6) (2011),
1767-1802.

\bibitem{Kunz85}
E.~Kunz, \emph{Introduction to commutative algebra and algebraic
geometry},
  Birkh{\"a}user, Boston, 1985.

\bibitem{Mezo20}
I.~Mez\"o, \emph{Combinatorics and number theory of counting
sequences}, Discrete Math. Appl. (Boca Raton), CRC Press, Boca
Raton, FL, 2020.

\bibitem{Pan01}
V.~Pan, \emph{Structured matrices and polynomials. {Unified}
superfast algorithms}, Birkh\"auser, Boston, 2001.

\bibitem{Rouillier97}
F.~Rouillier, \emph{Solving zero--dimensional systems through
rational
  univariate representation}, Appl. Algebra Engrg. Comm. Comput. \textbf{9} (1997),
  433--461.

\bibitem{Schost03}
E.~Schost, \emph{Computing parametric geometric resolutions}, Appl.
  Algebra Engrg. Comm. Comput. \textbf{13} (2003), 349--393.

\bibitem{Shafarevich94}
I.~Shafarevich, \emph{Basic algebraic geometry: {Varieties} in
projective
  space}, Springer, Berlin Heidelberg New York, 1994.


\end{thebibliography}
\end{document}